\documentclass[10pt]{amsart}
\usepackage{amsmath,amssymb,amsthm,array,hyperref,color,cite,graphicx}
\usepackage[all]{xypic}

\newtheorem{thm}{Theorem}[section]
\newtheorem{cor}[thm]{Corollary}
\newtheorem{lem}[thm]{Lemma}
\newtheorem{prop}[thm]{Proposition}
\theoremstyle{definition}

\newtheorem{df}[thm]{Definition}
\newtheorem{ex}[thm]{Example}

\newenvironment{rmk}
{\par\noindent\textbf{Remark.}\noindent}{}

\newcommand{\R}{\mathbb R}

\newcommand{\Q}{\mathbb Q}
\newcommand{\Sph}{\mathbb S}
\newcommand{\cat}[1]{\textup{\textbf{{#1}}}}

\newcommand{\Map}{\textup{Map}}

\newcommand{\id}{\textup{id}}

\newcommand{\colim}{\textup{colim}}
\newcommand{\hocolim}{\textup{hocolim}}
\newcommand{\holim}{\textup{holim}}
\newcommand{\ra}{\longrightarrow}
\newcommand{\la}{\longleftarrow}
\newcommand{\sma}{\wedge}
\newcommand{\barsmash}{\,\overline\wedge\,}

\newcommand{\simar}{\overset\sim\longrightarrow}
\newcommand{\congar}{\overset\cong\longrightarrow}

\newcommand{\mc}{\mathcal}
\newcommand{\ad}{\textup{Ad}}

\newcommand{\op}{\textup{op}}

\newcommand{\inj}{\hookrightarrow}
\newcommand{\cross}{\textup{cross}}
\newcommand{\cocross}{\textup{cocross}}
\newcommand{\fib}{\textup{fib}\,}
\newcommand{\cofib}{\textup{cofib}\,}
\newcommand{\hofib}{\textup{hofib}\,}

\newcommand{\tfib}{\textup{tfib}\,}
\newcommand{\tcofib}{\textup{tcofib}\,}
\newcommand{\il}{\underline{i}}
\newcommand{\jl}{\underline{j}}
\newcommand{\kl}{\underline{k}}
\newcommand{\nl}{\underline{n}}

\newcommand{\haut}{\textup{haut}}

\newcommand{\fin}{\textup{fin}}
\renewcommand{\uline}{\underline}

\title{A tower connecting gauge groups to string topology}
\author{Cary Malkiewich}

\begin{document}

\maketitle
\begin{abstract}
We develop a variant of calculus of functors, and use it to relate the gauge group $\mc G(\mc P)$ of a principal bundle $\mc P$ over $M$ to the Thom ring spectrum $(\mc P^\ad)^{-TM}$.
If $\mc P$ has contractible total space, the resulting Thom ring spectrum is $LM^{-TM}$, which plays a central role in string topology.
Cohen and Jones have recently observed that, in a certain sense, $(\mc P^\ad)^{-TM}$ is the linear approximation of $\mc G(\mc P)$.
We prove an extension of that relationship by demonstrating the existence of higher-order approximations and calculating them explicitly.
This also generalizes calculations done by Arone in \cite{arone_snaith}.
\end{abstract}

\parskip 0ex
\tableofcontents
\parskip 2ex

\section{Introduction}

If $M$ is a closed oriented manifold and $LM = \Map(S^1,M)$ is its free loop space, then the homology $H_*(LM)$ has a \emph{loop product} first described by Chas and Sullivan \cite{chas1999string}.
This loop product is homotopy invariant \cite{cohen2008homotopy} and has been calculated in a number of examples \cite{cohen2004loop}.
In \cite{felix2004monoid}, F\'elix and Thomas studied the loop product by defining a multiplication-preserving map
\begin{equation}\label{intro_felix}
H_*(\Omega_\id \haut(M);\Q) \ra H_{*+\dim M}(LM;\Q)
\end{equation}
where $\haut(M)$ is the space of self-homotopy equivalences of $M$, and the loops are based at the identity map of $M$.

In \cite{cohen2002homotopy}, Cohen and Jones described a ring spectrum $LM^{-TM}$ whose homology is $H_*(LM)$ but with a grading shift.
The multiplication on $LM^{-TM}$ gives the loop product on $H_*(LM)$, and the map of F\'elix and Thomas (\ref{intro_felix}) comes from a map of ring spectra
\begin{equation}\label{intro_2spectra}
\Sigma^\infty_+ \Omega_\id \haut M \ra LM^{-TM} 
\end{equation}
by taking rational homology groups.
In \cite{cohen2013gauge}, Cohen and Jones extend this map of ring spectra to a \emph{natural transformation of functors}
\begin{equation}\label{map_cohen_jones}
F \ra L
\end{equation}
\[ \begin{array}{c}
F,L: \mc R_M^\op \ra \mc Sp \\
F(M \amalg M) = \Sigma^\infty_+ \Omega_\id \haut M \\
L(M \amalg M) \simeq LM^{-TM}
\end{array} \]
Here $\mc R_M$ is the category of retractive spaces over $M$ and $\mc Sp$ is the category of spectra.
We will give these functors explicitly in section \ref{mapping}.
Both $F$ and $L$ are required to be \emph{homotopy functors}, meaning that they send equivalences of spaces to equivalences of spectra.
Cohen and Jones show that $L$ is the universal approximation of $F$ by an \emph{excisive} homotopy functor, i.e. one that takes each homotopy pushout square
\[ \xymatrix{
A \ar[r] \ar[d] & B \ar[d] \\
C \ar[r] & D } \]
to a homotopy pullback square
\[ \xymatrix{
L(A) & L(B) \ar[l] \\
L(C) \ar[u] & L(D) \ar[u] \ar[l] } \]
Such an $L$ takes finite sums of spaces to finite products of spectra.
This type of analysis is similar in spirit to Goodwillie's homotopy calculus of functors (\cite{calc1}, \cite{calc3}), though it is different in substance because the functors $F$ and $L$ are contravariant.
It is much more similar to the manifold calculus of Goodwillie and Weiss (\cite{weiss1999embeddings}, \cite{goodwillie1999embeddings}), though again it is subtly different because the functor $F$ has been defined on all spaces, and not just manifolds and embeddings.

Of course, in homotopy calculus one approximates a functor $F$ by an $n$-excisive functor $P_n F$ for each integer $n \geq 0$.
These fit into a tower
\[ F \ra \ldots \ra P_n F \ra \ldots \ra P_2 F \ra P_1 F \ra P_0 F \]
and one extracts information about $F$ from the layers
\[ D_n F := \hofib(P_n F \ra P_{n-1} F) \]
The map of functors (\ref{map_cohen_jones}) described by Cohen and Jones is only the first level of this tower:
\[ F \ra P_1 F \]
The main goal of this paper is to extend their construction by building the rest of the tower.
In order to do this, we develop a variant of homotopy calculus for contravariant functors from retractive spaces over $M$ to spectra.

In Definition \ref{n_excisive} we define $n$-excisive contravariant functors.
Our main results on $n$-excisive functors are Theorems \ref{final_second} and \ref{final_third}, which imply
\begin{thm}\label{intro_third}
Let $B$ be an unbased space. Consider a contravariant homotopy functor $F : \mc C^\op \ra \mc D$, where one of the following holds:
\begin{itemize}
\item $\mc C$ is the category of unbased finite CW complexes over $B$, and $\mc D$ is the category of based spaces or spectra.
\item $\mc C$ is the category of based finite CW complexes, and $\mc D$ is the category of based spaces or spectra.
\item $\mc C$ is the category of finite retractive CW complexes over $B$, and $\mc D$ is the category of spectra.
\end{itemize}
Then there exists a universal $n$-excisive approximation to $F$, called $P_n F$, and the natural transformation $F(X) \ra P_n F(X)$ is an equivalence when $X$ is a disjoint union of the initial object and $i$ discrete points, $0 \leq i \leq n$.
\end{thm}

\begin{rmk}
It has been pointed out to the author that the procedure found in \cite{de2012manifold} can be adapted to generalize manifold calculus from the category of manifolds to a broader category of topological spaces.
This method should also give results along the lines of Thm. \ref{intro_third}.
\end{rmk}

In section \ref{mapping} we explicitly define the functors of Cohen and Jones that extend the map of ring spectra (\ref{intro_2spectra}), and in section \ref{string_topology} we explicitly calculate the tower that extends the map of Cohen and Jones.
We explain in Proposition \ref{exists} why the above universal theorem is needed to conclude that our tower is correct.
Along the way to proving Theorems \ref{final_second} and \ref{final_third}, we prove a splitting result on homotopy limits in Proposition \ref{holim_twisted_split} that is reminiscent of a result of Dwyer and Kan (\cite{dwyerkanfunction}) on mapping spaces of diagrams.
This all implies the main result of the paper:

\begin{thm}\label{intro_string_tower}
There is a tower of homotopy functors
\[ F \ra \ldots \ra P_n F \ra \ldots \ra P_2 F \ra P_1 F \ra P_0 F \]
from finite retractive CW complexes over $M$ into spectra such that
\begin{enumerate}
\item $P_n F$ is the universal $n$-excisive approximation of $F$.
\item The map $F \ra P_1 F$ is the map (\ref{map_cohen_jones}) of Cohen and Jones.
\item $F(M \amalg M) \cong \Sigma^\infty_+ \Omega_\id \haut M$.
\item $P_1 F(M \amalg M) \simeq LM^{-TM}$.
\item $P_0 F(M \amalg M) = *$.
\item If $X$ is any finite retractive CW complex over $M$, the maps
\[ F(X) \ra P_n F(X) \ra P_{n-1}F(X) \]
are maps of ring spectra.
\item For all $n \geq 1$, $D_n F(M \amalg M)$ is equivalent to the Thom spectrum
\[ \mc C(LM;n)^{-TC(M;n)} \]
Here $C(M;n)$ is the space of unordered configurations of $n$ points in $M$, and $\mc C(LM;n)$ is the space of unordered collections of $n$ free loops in $M$ with distinct basepoints.
\end{enumerate}
\end{thm}

This linearization phenomenon is even more general.
Consider any principal bundle
\[ G \ra \mc P \ra M \]
The \emph{gauge group} $\mc G(\mc P)$ is defined to be the space of automorphisms of $\mc P$ as a principal bundle.
It is a classical fact that there is an associated \emph{adjoint bundle} $\mc P^\ad$, and that the gauge group $\mc G(\mc P)$ may be identified with the the space of sections $\Gamma_M(\mc P^\ad)$.

Gruher and Salvatore show in \cite{gruher2008generalized} that one may construct a Thom ring spectrum $(\mc P^\ad)^{-TM}$ out of the total space $\mc P^\ad$ of the adjoint bundle.
The multiplication on this ring spectrum gives a product on the homology $H_*(\mc P^\ad)$.
When the total space of $\mc P$ is contractible, the adjoint bundle $\mc P^\ad$ is equivalent to the free loop space $LM$, and the Gruher-Salvatore product on $H_*(\mc P^\ad)$ agrees with the Chas-Sullivan loop product on $H_*(LM)$.

In \cite{cohen2013gauge}, Cohen and Jones show that the map (\ref{intro_2spectra}) of ring spectra generalizes to a map of ring spectra
\begin{equation}\label{intro_2spectra_generalized}
\Sigma^\infty_+ \mc G(\mc P) \ra (\mc P^\ad)^{-TM} 
\end{equation}
Taking homology groups gives a multiplication-preserving map
\begin{equation}
H_*(\mc G(\mc P)) \ra H_{* + \dim M}(\mc P^\ad)
\end{equation}
which generalizes the map (\ref{intro_felix}) studied by F\'elix and Thomas.
Cohen and Jones extend this generalized map of ring spectra to a map of functors $F \ra L$ and show that $L$ is the universal approximation of $F$ by an excisive functor.
We extend their generalized result here as well:
\begin{thm}\label{intro_principal_tower}
There is a tower of homotopy functors
\[ F \ra \ldots \ra P_n F \ra \ldots \ra P_2 F \ra P_1 F \ra P_0 F \]
from finite retractive CW complexes over $M$ into spectra such that
\begin{enumerate}
\item $P_n F$ is the universal $n$-excisive approximation of $F$.
\item The map $F \ra P_1 F$ is the generalized map of Cohen and Jones.
\item $F(M \amalg M) \cong \Sigma^\infty_+ \mc G(\mc P)$.
\item $P_1 F(M \amalg M) \simeq (\mc P^\ad)^{-TM}$.
\item $P_0 F(M \amalg M) = *$.
\item If $X$ is any finite retractive CW complex over $M$, the maps
\[ F(X) \ra P_n F(X) \ra P_{n-1}F(X) \]
are maps of ring spectra.
\item For all $n \geq 1$, $D_n F(M \amalg M)$ is equivalent to the Thom spectrum
\[ \mc C(\mc P^\ad ;n)^{-TC(M;n)} \]
where $\mc C(\mc P^\ad ;n)$ is the space of unordered configurations of $n$ points in the total space $\mc P^\ad$ which have distinct images in $M$.
\end{enumerate}
\end{thm}

The outline of the paper is as follows.
In Section 2, we define $n$-excisive functors and give criteria for recognizing the universal $n$-excisive approximation $P_n F$ of a given functor $F$.
In Section 3, we give a detailed construction of a tower which generalizes the above two examples.
In Section 4, we specialize to the above two examples and do some computations.
Sections 5-8 supply a missing ingredient from the previous sections: a functorial construction of $P_n F$ for general $F$.
This material may be of independent interest in the general study of calculus of functors.

The author would like to acknowledge Greg Arone, Boris Chorny, Ralph Cohen, John Klein, and Peter May for many enlightening ideas and helpful conversations in the course of this project.
This paper represents a part of the author's Ph.D. thesis, written under the direction of Ralph Cohen at Stanford University.

\section{Excisive functors}

Fix an unbased space $B$.
Like every space that follows, we assume it is compactly generated and weak Hausdorff.

\begin{df}\label{cats}
\begin{itemize}
\item Let $\mc S_B$ be the category of spaces over $B$.
The objects are spaces $X$ equipped with maps $X \ra B$.
Define two subcategories
\[ \mc S_{B,n} \subset \mc S_{B,\fin} \subset \mc S_B \]
as follows.
The subcategory $\mc S_{B,\fin}$ consists of all finite CW complexes over $B$.
The subcategory $\mc S_{B,n}$ consists of discrete spaces with at most $n$ points over $B$.
For simplicity, we may as well assume that the finite CW complexes are embedded in $B \times \R^\infty$, and that $\mc S_{B,n}$ has only one space $\il = \{1,\ldots,i\}$ for each $0 \leq i \leq n$ and each map $\il \ra B$.
\item Let $\mc R_B$ be the category of spaces containing $B$ as a retract.
As before, define two subcategories
\[ \mc R_{B,n} \subset \mc R_{B,\fin} \subset \mc R_B \]
where $\mc R_{B,\fin}$ consists of spaces $X$ for which $(X,B)$ is a finite relative CW complex, and $\mc R_{B,n}$ consists of spaces of the form $\il \amalg B$, $\il = \{1,\ldots,i\}$, for $0 \leq i \leq n$.
\end{itemize}
\end{df}

Of course, if $B = *$ then $\mc S_B$ and $\mc R_B$ are the familiar categories of unbased spaces $\mc U$ and based spaces $\mc T$, respectively.
The following definition should be seen as an analogue of Goodwillie's notion of \emph{$n$-excisive} for covariant functors \cite{calc2}:

\begin{df}\label{n_excisive}
A contravariant functor $\mc R_B^\op \overset{F}\ra \mc T$ is \emph{$n$-excisive} if
\begin{itemize}
\item $F$ is a \emph{homotopy functor}, meaning weak equivalences $X \simar Y$ of spaces containing $B$ as a retract are sent to weak equivalences $F(Y) \simar F(X)$ of based spaces.
\item $F$ takes \emph{strongly co-Cartesian cubes} (equivalently, \emph{pushout cubes}) of dimension at least $n+1$ to \emph{Cartesian cubes} (see \cite{calc2}).
When $n=1$, this means that $F$ takes homotopy pushout squares to homotopy pullback squares.
\item $F$ is \emph{finitary}, meaning that it sends filtered homotopy colimits to homotopy limits.
In particular, $F$ is determined up to equivalence by its behavior on relative finite CW complexes $B \inj X$.
\end{itemize}
\end{df}

This definition is easily modified to suit many cases.
When restricting to finite CW complexes ($\mc R_{B,\fin}$), we drop the last condition.
If $\mc Sp$ denotes a suitable model category of spectra, for example the category of prespectra described in \cite{mmss}, then a contravariant functor $\mc R_B^\op \overset{F}\ra \mc Sp$ is \emph{$n$-excisive} if it satisfies the above properties, with ``equivalence of based spaces'' replaced by ``stable equivalence of spectra.''
For most models of spectra, we may post-compose $F$ with a fibrant replacement functor, and conclude that $F$ is an $n$-excisive functor to spectra iff the functor $F_j$ at each spectrum level is an $n$-excisive functor to based spaces.
It is also straightforward to define \emph{$n$-excisive} for functors from unbased spaces $\mc S_B$ or unbased finite spaces $\mc S_{B,\fin}$ to either spaces $\mc T$ or spectra $\mc Sp$.

If $F$ is a contravariant $n$-excisive functor, then $F$ is completely determined by its values on the discrete spaces with at most $n$ points:

\begin{prop}\label{polys_determined_by_nplusone_points}
\begin{itemize}
\item Let $F$ and $G$ be $n$-excisive functors $\mc R_B^\op \ra \mc T$, and $F \overset\eta\ra G$ a natural transformation.
If $\eta$ is an equivalence when restricted to the subcategory $\mc R_{B,n}^\op$, then $\eta$ is also an equivalence on the rest of $\mc R_B^\op$.
\item If $F$ and $G$ are $n$-excisive functors $\mc S_B^\op \ra \mc T$, and $F \overset\eta\ra G$ is an equivalence on $\mc S_{B,n}^\op$, then $\eta$ is also an equivalence on $\mc S_B^\op$.
\item The obvious analogues hold when the source of $F$ and $G$ is instead the category of finite CW complexes $\mc S_{B,\fin}$ or $\mc R_{B,\fin}$, or when the target is spectra $\mc Sp$ instead of spaces $\mc T$.
\end{itemize}
\end{prop}


\begin{proof}
We will only prove the first statement, by induction on the dimension of the relative CW complex $B \ra X$.
The key fact is that a map of Cartesian cubes is an equivalence on the initial vertex if it is an equivalence on all the others.

Given a 0-dimensional complex of the form $\uline{m} \amalg B$ with $m \geq n+1$, we may construct a pushout $(n+1)$-cube whose final vertex is $\uline{m} \amalg B$, and all other vertices are of the form $\kl \amalg B$ for varying $k \leq m$.
Applying $F$ and $G$ to this pushout cube gives two Cartesian cubes, and $\eta$ gives a map between the two Cartesian cubes.
Inductively, this map is an equivalence on every vertex but the initial one, so it is an equivalence on the initial vertex as well, giving $F(\uline{m} \amalg B) \simar G(\uline{m} \amalg B)$.

To do higher-dimensional complexes, it is necessary to attach the $d$-dimensional cells to a $(d-1)$-dimensional complex in $(n+1)$ stages, so that only by taking an $(n+1)$-fold pushout do we obtain a $d$-dimensional space.
To accomplish this, we define a \emph{layer cake space} $L^d_T$ for each subset $T \subset \{1,2,\ldots,n\}$.
$L^d_T$ is the subspace of the closed $d$-dimensional unit cube $I^d = [0,1]^d$ consisting of those points whose final coordinate is in the set
\[ \left\{0,\frac1n,\frac2n,\ldots,1\right\} \cup \{t : \lceil nt \rceil \in T\} \]
So $L^d_{\{1,\ldots,n\}}$ is the entire cube, while $L^d_\emptyset$ is homotopy equivalent to $n+1$ copies of $I^{d-1}$ glued along their boundaries.
Intuitively, $L^d_T$ consists of all the frosting in a layer cake, together with a selection of layers given by $T$.
If $T$ is a proper subset, then $L^d_T$ is homotopy equivalent rel $\partial I^d$ to $\partial I^d$ with some $(d-1)$-cells attached.

Now assume that $\eta$ is an equivalence on all finite $(d-1)$-dimensional complexes.
Let $X$ be a $d$-dimensional finite complex, with top-dimensional attaching maps $\{\partial I^d \overset{\varphi_\alpha}\ra X^{(d-1)}\}_{\alpha \in A}$.
Form an $(n+1)$-dimensional pushout cube with the following description.
The initial vertex is $\coprod_A L^d_\emptyset$, a disjoint union of one empty layer cake for each $d$-cell of $X$.
Next, let $n$ of the $n+1$ adjacent vertices be $\coprod_A L^d_{\{i\}}$ as $i$ ranges over $\{1,\ldots,n\}$.
Finally, let the last adjacent vertex be the pushout of $X^{(d-1)}$ and $\coprod_A L^d_\emptyset$ along $\coprod_A \partial I^d$.
Then the final vertex of our pushout cube is homeomorphic to $X$, while every vertex other than the final one is homotopy equivalent to a $(d-1)$-dimensional cell complex.
After applying $F$ and $G$, $\eta$ gives us a map between two Cartesian cubes, and the map is an equivalence on every vertex but the initial one.
So $F(X) \overset\eta\ra G(X)$ is an equivalence as well, completing the induction.

Of course, if the source category of $F$ and $G$ has infinite CW complexes, we express each CW complex as a filtered homotopy colimit of its finite subcomplexes and invoke the colimit axiom.
To move to all spaces, we recall that $F$ and $G$ preserve weak equivalences, and that every space over $B$ has a functorial CW approximation.
\end{proof}

Now suppose $F$ is a contravariant homotopy functor.
We want to define a ``best possible'' approximation of $F$ by an $n$-excisive functor.
By this we mean an $n$-excisive functor $P_n F$ with the same source and target as $F$, and a natural transformation $F \ra P_n F$ that is universal among all maps $F \ra P$ from $F$ into an $n$-excisive functor $P$:
\[ \xymatrix{
F \ar[d] \ar[r] & P \\
P_n F \ar@{-->}[ru]_-{!} & } \]
We will relax this condition to take place in the \emph{homotopy category} of functors.
Following \cite{calc3}, we get this homotopy category by formally inverting the following equivalences:

\begin{df}
An \emph{equivalence} of functors is a natural transformation $F \ra G$ that yields equivalences $F(X) \simar G(X)$ for all spaces $X$.
\end{df}

Unfortunately, this homotopy category of functors has significant set-theory issues.
First of all, the category of all functors from spaces to spaces is not really a category in the usual sense.
This is because when we choose two functors $F$ and $G$, the collection of all natural transformations $F \ra G$ forms a proper class.
In other words, the category of functors has large hom-sets.
The homotopy category of functors has even larger hom-sets \cite{calc3}.

One way of resolving this issue is to restrict to \emph{small} functors as defined in \cite{chorny2006homotopy}.
The small functors form a model category, so their homotopy category has small hom-sets.

We will use a different fix, since we are ultimately interested in a result about compact manifolds.
We will restrict our attention to functors defined on finite CW complexes ($\mc S_{B,\fin}$ or $\mc R_{B,\fin}$) instead of all spaces ($\mc S_B$ or $\mc R_B$).
Finite CW complexes over $B$ can always be embedded into $B \times \R^\infty$, so we can easily make $\mc S_{B,\fin}$ and $\mc R_{B,\fin}$ into small categories.
Then the category of functors from $\mc S_{B,\fin}$ or $\mc R_{B,\fin}$ into spaces or spectra has the projective model structure, as discussed below in section \ref{cells}.
Therefore our homotopy category of functors has small hom-sets.

Now that we are on solid footing, we can return to the problem of finding a universal $n$-excisive approximation $P_n F$ to the homotopy functor $F$.
It turns out that $P_n F$ will actually \emph{agree} with $F$ on the spaces with at most $n$ points.
This is similar to manifold calculus (\cite{weiss1999embeddings}, \cite{goodwillie1999embeddings}) but quite different from the case of covariant functors (\cite{calc3}).
To be more explicit, in Goodwillie calculus one builds a Taylor series, giving a tower of functors that converges to $F$ on highly connected spaces.
By contrast, our theory will build a polynomial interpolation.
We sample our functor $F$ at $(n+1)$ particular homotopy types $\uline{0},\ldots,\uline{n}$ and then we build the unique degree $n$ polynomial $P_n F$ that has the same values on those $(n+1)$ homotopy types.
As a result, our tower of functors will converge to $F$ on low dimensional spaces, as explained in Section \ref{convergence} below.

So let $F$ be any contravariant homotopy functor from finite CW complexes ($\mc R_{B,\fin}$ or $\mc S_{B,\fin}$) to either based spaces or spectra.
(There is one exception to this setup, as explained in section \ref{secondconstruction}.)
In sections \ref{firstconstruction}, \ref{secondconstruction}, and \ref{hack} below we will define a functor $P_n F$ with the same source and target as $F$, and a natural transformation $p_n F: F \ra P_n F$, both functorial in $F$.
Then we will show two things:
\vspace{-6pt}
\begin{itemize}
\item $P_n F$ is $n$-excisive.
\item $F \ra P_n F$ is an equivalence on $\mc R_{B,n}^\op$ (based case) or $\mc S_{B,n}^\op$ (unbased case).
\end{itemize}

\begin{prop}\label{exists}
If $F \ra P_n F$ is a functorial construction satisfying the above properties, then $P_n F$ is universal among all $n$-excisive $P$ with natural transformations $F \ra P$ in the homotopy category of functors.
\end{prop}

\begin{proof}
Easy adaptation of (\cite{calc3}, 1.8).
\end{proof}

\begin{cor}[Recognition Principle for $P_n F$]\label{recognition}
Given that such a construction $P_n$ exists, if $P$ is any $n$-excisive functor with a map $F \ra P$ that is an equivalence on $\mc R_{B,n}^\op$ or $\mc S_{B,n}^\op$, then $P$ is canonically equivalent to $P_n F$.
\end{cor}

\begin{proof}
By the universal property of $P_n F$ there exists a unique map $P_n F \ra P$, but this is a map of $n$-excisive functors and an equivalence on $\mc R_{B,n}^\op$ or $\mc S_{B,n}^\op$, so it is an equivalence of functors.
\end{proof}

\begin{rmk}
This recognition argument applies equally well to the case of covariant functors from spaces to spectra.
This is alarming, because in that case $F \ra P_n F$ is usually not an equivalence on the spaces with at most $n$ points.
The only possible conclusion is that, in such a setting, there is no construction $P_n$ satisfying the above properties.
\end{rmk}

We will delay the construction of $P_n F$ to section \ref{firstconstruction}.
In the next section, we will apply this recognition theorem in a particular example.

\section{The tower of approximations of a mapping space}\label{mapping}

Now we will compute the tower of universal $n$-excisive approximations of the functor
\[ F(X) = \Sigma^\infty \Map_B(X,E) \]
from retractive spaces over $B$ to spectra.
The map of Cohen and Jones described in the introduction is the special case $X = M \amalg M$, $B = M$, and $E = LM \amalg M$.
Our results in this section are proven using techniques from model categories, so we will fix some notation for this following \cite{mmss} and \cite{ms}.

\begin{df}
Let $X$ and $Y$ be unbased spaces over $B$, or retractive spaces over $B$.
\begin{itemize}
\item A $q$-cofibration $X \ra Y$ is a retract of a relative cell complex.
\item A finite $q$-cofibration $X \ra Y$ is a retract of a finite relative cell complex.
\item A $q$-fibration $X \ra Y$ is a Serre fibration.
\item An $h$-cofibration $X \ra Y$ is a map of spaces satisfying the homotopy extension property.
\item An $h$-fibration $X \ra Y$ is a Hurewicz fibration.
\end{itemize}
\end{df}

We should also be precise about our definition of $F(X) = \Sigma^\infty \Map_B(X,E)$.

\begin{df}
\begin{itemize}
\item If $E$ is a retractive space over $B$, let $\Sigma_B E$ denote the \emph{fiberwise reduced suspension} of $E$.
\item An \emph{ex-fibration} is a retractive space $E$ over $B$ such that $E \ra B$ is a Hurewicz fibration, and $B \ra E$ is well-behaved in a sense described in (\cite{ms}, 8.2).
For our purposes, the most important property of an ex-fibration $E$ is that the fiberwise reduced suspension $\Sigma_B E$ is again an ex-fibration.
\item If $X$ is a finite $q$-cofibrant retractive space over $B$, and $E$ is an ex-fibration over $B$, let $\Map_B(X,E)$ denote the space of maps $X \ra E$ respecting the maps into and out of $B$.
We may grow a whisker on this space to ensure that it is well-based.
\item Similarly, let $\Map_B(X, \Sigma^\infty_B E)$ denote a spectrum whose $k$th level is fiberwise maps from $X$ into $\Sigma^k_B E$.
\end{itemize}
\end{df}

Now we will build up to the definition of the functors that approximate $F$.
Let $\cat M_n$ be the category whose objects are the finite unbased sets $\uline{0} = \emptyset,$ $\uline{1} = *,$ $\uline{2} = \{1,2\},$ $\ldots,$ $\uline{n} = \{1,\ldots,n\}$ and whose morphisms are the \emph{surjective} maps.
For any space $X$, we can form a diagram of unbased spaces indexed by the opposite category $\cat M_n^\op$:
\[ \il \mapsto X^i = \Map(\il,X) \]
Algebraically, this diagram is the functor represented by $X$.
Geometrically, this is a diagram of generalized diagonal maps: each map $\il \ra \uline{i-1}$ results in an inclusion $X^{i-1} \ra X^i$ whose image consists of those $i$-tuples in which a particular pair of coordinates is repeated.
The union of all such images the \emph{fat diagonal}, which we will denote
\[ \Delta \subset X^i \]

\begin{df}\label{smash}
Let $X$ be a finite $q$-cofibrant retractive space over $B$ and let $E$ be an ex-fibration over $B$.
\begin{itemize}
\item Let $X \barsmash X$ be the \emph{external smash product} of $X$ with itself; this is a retractive space over $B \times B$ whose fiber over $(b_1,b_2)$ is the smash product of the fibers $X_{b_1} \sma X_{b_2}$.
More generally, $X^{\barsmash n}$ is the $n$-fold iterated external smash product, which is a retractive space over $B^n$.
\item Define
\[ \Map_{(\cat M_n^\op,\{B^i\})}(X^{\barsmash i}, \Sigma^\infty_{B^i} E^{\barsmash i}) \]
to be the spectrum whose $k$th level is collections of maps of retractive spaces
\[ \xymatrix{
\Sigma^k * & \Sigma^k_B E & \ldots & \Sigma^k_{B^n} E^{\barsmash n} \\
{*} \ar[u]_-{f_0} & X \ar[u]_-{f_1} & \ldots & X^{\barsmash n} \ar[u]_-{f_n} } \]
such that each surjective map $\il \la \jl$ gives a commuting square
\[ \xymatrix{
\Sigma^k_{B^i} E^{\barsmash i} \ar[r] & \Sigma^k_{B^j} E^{\barsmash j} \\
X^{\barsmash i} \ar[r] \ar[u]_-{f_i} & X^{\barsmash j} \ar[u]_-{f_j} } \]
\end{itemize}
\end{df}

\begin{rmk}
Note that the collection of maps $(f_0,f_1,\ldots,f_n)$ is completely determined by the last map $f_n$, which must be $\Sigma_n$-equivariant.
When $n \geq 3$, not every $\Sigma_n$-equivariant map arises this way.
\end{rmk}

\begin{rmk}
One might expect $S^0 \ra \Sigma^k S^0$ in the place of $* \ra \Sigma^k *$, since an empty smash product is $S^0$.
This answer would give the approximation to the functor $F \vee \Sph$ instead of $F$.
A similar phenomenon happens in Cor. \ref{binomial} below.
\end{rmk}

To summarize, for any unbased space $B$ and ex-fibration $E \ra B$, we have defined a tower of functors on the category $\mc R_{B,\fin}$:
\[ \begin{array}{ccl}
F(X) &=& \Sigma^\infty \Map_B(X,E) \\
\downarrow && \\
\vdots && \vdots  \\
P_n F(X) &=& \Map_{(\cat M_n^\op,\{B^i\})}(X^{\barsmash i}, \Sigma^\infty_{B^i} E^{\barsmash i}) \\
\vdots && \vdots  \\
\downarrow && \\
P_2 F(X) &=& \Map_{B \times B}(X \barsmash X,\Sigma^\infty_{B \times B} E \barsmash E)^{\Sigma_2} \\
\downarrow && \\
P_1 F(X) &=& \Map_B(X,\Sigma^\infty_B E) \\
\downarrow && \\
P_0 F(X) &=& *
\end{array} \]
We will justify the notation with Theorem \ref{main}, which shows that $P_n F(X)$ is the universal $n$-excisive approximation to $F(X)$.
This is a generalization of an observation made by Greg Arone about the tower in \cite{arone_snaith}.

\begin{rmk}
It would also be natural to examine the functor $X \mapsto \Map_B(X,E)$, before applying $\Sigma^\infty$ to it.
This functor is already 1-excisive, so it does not give an interesting tower.
It is also natural to consider
\[ \hat F(X) = \Sigma^\infty \Map_B(X,E) \]
for unbased $X$ over $B$, without a basepoint section.
But $\hat F(X) = F(X \amalg B)$, so $P_n \hat F(X) = P_n F(X \amalg B)$ by comparing universal properties.
\end{rmk}

\subsection{Cell complexes of diagrams}\label{cells}

Many of the proofs that follow rely on the same basic idea: we start with a diagram of spaces or spectra that is built inductively out of cells, and we define maps of diagrams one cell at a time.
In doing so, we are using the following standard facts.
First, both spaces and spectra have cofibrantly generated model structures \cite{mmss}.
Therefore the category of diagrams indexed by $\cat I$ can be endowed with the \emph{projective model structure}.
The weak equivalences $F \ra G$ of diagrams are the maps that give objectwise equivalences $F(i) \simar G(i)$, and the fibrations are the objectwise ($q$-)fibrations.
The projective model structure is again cofibrantly generated.

To understand the cofibrant diagrams, define a functor that takes a based space (or spectrum) $X$ and produces the diagram
\[ F_i(X)(j) = \cat I(i,j)_+ \sma X \]
A map of diagrams $F_i(X) \ra G$ is the same as a map of spaces (or spectra) $X \ra G(i)$.
This property is clearly useful for defining maps of diagrams one cell at a time.
We can define a \emph{diagram cell} by applying $F_i$ to the maps $S^{n-1}_+ \ra D^n_+$, and then define a \emph{diagram cell complex} to be any iterated pushout along coproducts of diagram cells.
Every diagram cell complex is cofibrant; in fact, the cofibrations are just the retracts of the relative cell complexes.

Now we will prove that one particular diagram is cofibrant.
Recall that $\cat M_n$ is the category with one object $\il = \{1,\ldots,i\}$ for each integer $0 \leq i \leq n$, with maps $\il \ra \jl$ the surjective maps of sets.
The maps are not required to preserve ordering, so in particular $\cat M_n(\il,\il) \cong \Sigma_i$, the symmetric group on $i$ letters.

\begin{prop}
If $X$ is a based cell complex, then $\{X^{\sma i}\}_{i=0}^n$ is a cell complex of $\cat M_n^\op$ diagrams.
Similarly for Cartesian products $\{X^i\}$.
If $X$ is $q$-cofibrant then $\{X^{\sma i}\}$ and $\{X^i\}$ are cofibrant diagrams.
\end{prop}
\begin{proof}
It suffices to do the case where $X$ is a cell complex.
We put a new cell complex structure on $X^{\sma n}$ so that the fat diagonal is a subcomplex, and every cell outside of the fat diagonal is permuted freely by the $\Sigma_n$-action.
This reduces quickly to the case where $X$ has a single cell.

The product $\prod^n I^m \cong \prod^n [0,1]^m$ may be identified with the space of all $n \times m$ matrices, with real entries between 0 and 1.
The $\Sigma_n$ action permutes the rows.
Within this space, we define an open simplex of dimension $d$ for each partition of the $nm$ entries of the matrix into $d$ nonempty equivalence classes, along with a choice of total ordering on the equivalence classes.
This simplex corresponds to the subspace of matrices for which the equivalent entries have the same value, and the values are ordered according to the chosen total ordering.

The closures of these simplices give a triangulation of the cube $\prod^n I^m \cong \prod^n [0,1]^m$.
Each generalized diagonal is defined by setting an equivalence relation on the rows of the matrix, and requiring that equivalent rows have the same values.
This is clearly an intersection of conditions we used to define the simplices above, so each generalized diagonal is a union of simplices.
In addition, the simplices off the fat diagonal are freely permuted by the $\Sigma_n$ action.
This finishes the proof.
\end{proof}

\begin{prop}
If $X$ is a based cell complex and $A$ is a subcomplex then $\{A^{\sma i}\} \ra \{X^{\sma i}\}$ is a relative cell complex of $\cat M_n^\op$ diagrams.
If $* \ra A \ra X$ are $q$-cofibrations then $\{A^{\sma i}\} \ra \{X^{\sma i}\}$ is a cofibration of $\cat M_n^\op$ diagrams.
\end{prop}
\begin{proof}
Each cell of $X^{\sma i}$ lying outside $A^{\sma i}$ is a product of cells in $X$, at least one of which is not a cell in $A$.
As above, we subdivide each of these cells so that $\Delta \cup A^{\sma i}$ is a subcomplex when $\Delta$ is any of the generalized diagonals.
Off the fat diagonal, the $\Sigma_i$ action still freely permutes the cells.
This gives the recipe for building the map of diagrams $\{A^{\sma i}\} \ra \{X^{\sma i}\}$ out of free cells of diagrams.

Suppose that $* \ra A \ra X$ are $q$-cofibrations, and we want to show that $\{A^{\sma i}\} \ra \{X^{\sma i}\}$ is a cofibration of diagrams.
Then without loss of generality we can replace $A \ra X$ by a relative cell complex $A \ra X'$.
Then we can replace $* \ra A$ by a relative cell complex $* \ra A'$, and we get the sequence of relative cell complexes $* \ra A' \ra X' \cup_A A'$ containing $* \ra A \ra X$ as a retract.
Then we apply the same argument as above.
\end{proof}

\begin{prop}\label{diagonalcells}
If $X$ is a retractive cell complex over $B$ then $\{B^i\} \ra \{X^{\barsmash i}\}$ is a relative cell complex of $\cat M_n^\op$ diagrams.
If $B \ra A \ra X$ are $q$-cofibrations over $B$ then $\{A^{\barsmash i}\} \ra \{X^{\barsmash i}\}$ is a cofibration of $\cat M_n^\op$ diagrams.
\end{prop}
\begin{proof}
We must verify that $B^i \inj X^{\barsmash i}$ is a relative cell complex with one cell for each $i$-tuple of relative cells of $B \inj X$.
This is a straightforward adaptation of standard arguments, but it is worth pointing out that these arguments derail if we don't work in the category of compactly generated weak Hausdorff spaces.
Once we are assured that everything is a cell complex, the rest of the proof follows as above.
\end{proof}

Recall that an \emph{acyclic cofibration} (of spaces, spectra, or diagrams) is a map that is both a cofibration and a weak equivalence.

\begin{cor}
If $* \ra A \ra X$ are $q$-cofibrations and $A \ra X$ is acyclic then $\{A^{\sma i}\} \ra \{X^{\sma i}\}$ is an acyclic cofibration of $\cat M_n^\op$ diagrams.
Similarly for Cartesian products.
\end{cor}
\begin{proof}
Since $A$ and $X$ are $q$-cofibrant they are well-based, meaning $* \ra A$ is an $h$-cofibration.
Therefore since $A \ra X$ is a weak equivalence, $A^{\sma i} \ra X^{\sma i}$ is a weak equivalence as well.
\end{proof}

\begin{cor}\label{diagonalacycliccells}
Each acyclic cofibration $A \ra X$ of $q$-cofibrant retractive spaces over $B$ induces a acyclic cofibration of $\cat M_n^\op$ diagrams $\{A^{\barsmash i}\} \ra \{X^{\barsmash i}\}$.
\end{cor}
\begin{proof}
Again, we just need to show that $A^{\barsmash i} \ra X^{\barsmash i}$ is a weak equivalence of total spaces.
Suppose that $i = 2$.
Let $H_A$ be the homotopy pushout of
\[ \xymatrix{
B \times B & \\
(A \times B) \cup_{B \times B} (B \times A) \ar[u] \ar[r] & A \times A } \]
Then $H_A$ is equivalent to the strict pushout $A \barsmash A$, because the bottom map is an $h$-cofibration.
This gives a square
\[ \xymatrix{
H_A \ar[r]^-\sim \ar[d]^-\sim & H_X \ar[d]^-\sim \\
A \barsmash A \ar[r] & X \barsmash X } \]
from which we see that the bottom map is an equivalence.
When $i > 2$, simply replace one of the two copies of $A$ by the space $A^{\barsmash (i-1)}$.
\end{proof}

\subsection{Proof that the tower is correct}\label{correct}

As in Def. \ref{smash} above, let $X$ be a finite $q$-cofibrant retractive space over $B$, and let $E$ be an ex-fibration over $B$.

\begin{prop}
$F(X) = \Sigma^\infty \Map_B(X,E)$ is a homotopy functor on the category of $q$-cofibrant retractive spaces over $B$.
\end{prop}
\begin{proof}
We prove $F(X)$ takes weak equivalences between $q$-cofibrant retractive spaces over $B$ to level equivalences of spectra.
Since all of our mapping spaces are well-based, it suffices to prove this fact for the functor $\Map_B(X,E)$.
By Ken Brown's lemma, it suffices to take an acyclic $q$-cofibration $X \ra Y$ and show that
\[ \Map_B(Y,E) \ra \Map_B(X,E) \]
is a weak equivalence.
So take the square of based spaces
\[ \xymatrix{
S^{n-1}_+ \ar[r] \ar[d] & \Map_B(Y,E) \ar[d] \\
D^n_+ \ar[r] \ar@{-->}[ur] & \Map_B(X,E) } \]
The right-hand vertical map is a weak equivalence if we can show the desired lift always exists.
This is equivalent to finding a lift in the square
\[ \xymatrix{
(S^{n-1} \times Y) \cup (D^n \times X) \ar[r] \ar[d] & E \ar[d] \\
D^n \times Y \ar[r] \ar@{-->}[ur] & B } \]
Since $E \ra B$ is a $q$-fibration, it suffices to show that the left-hand vertical map is an acyclic $q$-cofibration.
This is the main axiom for checking that (compactly generated) spaces form a monoidal model category.
It is proven by a standard argument found in \cite{hovey1999model}, so we are done.
Alternatively, the homotopy invariance of mapping spaces from cofibrant objects to fibrant objects could also be deduced from the results of Dwyer and Kan on hammock localization \cite{dwyerkanfunction}.
\end{proof}

\begin{prop}
$P_n F(X) = \Map_{(\cat M_n^\op,\{B^i\})}(X^{\barsmash i}, \Sigma^\infty_{B^i} E^{\barsmash i})$ as in Def \ref{smash} is also a homotopy functor on the category of $q$-cofibrant retractive spaces over $B$.
\end{prop}
\begin{proof}
Again, by Ken Brown's lemma it suffices to take an acyclic $q$-cofibration $X \ra Y$ and show that
\[ \Map_{(\cat M_n^\op,\{B^i\})}(Y^{\barsmash i}, \Sigma^\infty_{B^i} E^{\barsmash i})
\ra \Map_{(\cat M_n^\op,\{B^i\})}(X^{\barsmash i}, \Sigma^\infty_{B^i} E^{\barsmash i}) \]
is a level equivalence of spectra.
So we take any square of spaces
\[ \xymatrix{
S^{n-1}_+ \ar[r] \ar[d] & \Map_{(\cat M_n^\op,\{B^i\})}(Y^{\barsmash i}, \Sigma^k_{B^i} E^{\barsmash i}) \ar[d] \\
D^n_+ \ar[r] \ar@{-->}[ur] & \Map_{(\cat M_n^\op,\{B^i\})}(X^{\barsmash i}, \Sigma^k_{B^i} E^{\barsmash i}) } \]
and show the dotted diagonal map exists.
This is equivalent to a lift in this square of diagrams indexed by $\cat M_n^\op$:
\[ \xymatrix{
(S^{n-1}_+ \barsmash Y^{\barsmash i}) \cup (D^n_+ \barsmash X^{\barsmash i}) \ar[r] \ar[d] & \Sigma^k_{B^i} E^{\barsmash i} \ar[d] \\
D^n_+ \barsmash Y^{\barsmash i} \ar[r] \ar@{-->}[ur] & B^i } \]
Since we assumed that $E \ra B$ was an ex-fibration, the right-hand vertical map is an ex-fibration as well (\cite{ms}, 8.2.4).
Therefore it is also a $q$-fibration.
So it suffices to show the left-hand vertical map is an acyclic cofibration of diagrams.
By standard techniques, this reduces to showing that
\[ X^{\barsmash i} \ra Y^{\barsmash i} \]
is an acyclic cofibration of diagrams, which we proved in Prop. \ref{diagonalacycliccells} above.
\end{proof}

\begin{prop}
The natural map $F \ra P_n F$ is an equivalence on any retractive space of the form $\il \amalg B$ when $0 \leq i \leq n$.
\end{prop}
\begin{proof}
When $X = \il \amalg B$, the fat diagonal covers all of $X^{\barsmash j}$ for $j > i$.
Therefore a natural transformation of $\cat M_n^\op$-diagrams is determined by what it does on $X^{\barsmash i} = \il^i \amalg B^i$.
This is an $i^i$-tuple of points in various fibers of $\Sigma^\infty_{B^i} E^{\barsmash i}$, with compatibility conditions.
The compatibility conditions force us to have only one point for each nonempty subset of $\il$.
Therefore the map $F(\il \amalg B) \ra P_n F(\il \amalg B)$ becomes
\[ \Sigma^\infty (E_{b_1} \times \ldots \times E_{b_i})
\ra \prod_{S \subset \underline i,\, S \neq \emptyset} \Sigma^\infty \bigwedge_{s \in S} E_{b_s} \]
From Cor. \ref{binomial} below, this map is always an equivalence.
\end{proof}

\begin{prop}
$P_n F$ is $n$-excisive.
\end{prop}
\begin{proof}
Let $S$ be a finite set with $|S| \geq n+1$, and take any strongly co-Cartesian cube of retractive spaces over $B$ indexed by the subsets of $S$.
This cube is equivalent to a cube of pushouts along relative CW complexes
\[ A \ra X_s \qquad s \in S \]
of retractive spaces over $B$.
Applying $P_n F$, we get a cube of spectra which sends each subset $T \subseteq S$ to the spectrum
\[ P_n F\left(\bigcup_{t \in T} X_t \right) \]
Let's show that this cube is level Cartesian.
This will be enough to prove that the cube is actually Cartesian, since finite homotopy limits commute with fibrant replacement of spectra.

Fix a nonnegative integer $k$ and restrict attention to level $k$ of the spectra in the cube.
This turns out to be a \emph{fibration cube} as defined in \cite{calc2}.
To prove this, we have to construct this lift for any space $K$ and any subset $S' \subseteq S$:
\[ \xymatrix{
K \ar[r] \ar[d] & \Map_{(\cat M_n^\op,\{B^i\})}\left(\left( \bigcup_{s \in S'} X_s \right)^{\barsmash i}, \Sigma^k_{B^i} E^{\barsmash i} \right) \ar[d] \\
K \times I \ar[r] \ar@{-->}[ur] & \underset{T \subsetneq S'}\lim\, \Map_{(\cat M_n^\op,\{B^i\})}\left(\left( \bigcup_{t \in T} X_t \right)^{\barsmash i}, \Sigma^k_{B^i} E^{\barsmash i} \right) \ar@{=}[d] \\
& \Map_{(\cat M_n^\op,\{B^i\})}\left(\underset{T \subsetneq S'}\colim \left( \bigcup_{t \in T} X_t \right)^{\barsmash i}, \Sigma^k_{B^i} E^{\barsmash i} \right) } \]
Rearranging gives
\[ \xymatrix{
K \times \left[ \left( \{0\} \times \left( \bigcup_{s \in S'} X_s \right)^{\barsmash i} \right) \cup \left( I \times \underset{T \subsetneq S'}\colim \left( \bigcup_{t \in T} X_t \right)^{\barsmash i} \right) \right] \ar[r] \ar[d] & \Sigma^k_{B^i} E^{\barsmash i} \ar[d] \\
K \times I \times \left( \bigcup_{s \in S'} X_s \right)^{\barsmash i} \ar[r] \ar@{-->}[ur] & B^i } \]
This is a square of maps of diagrams. The left and right vertical maps are at each level of the diagram $h$-cofibrant and $h$-fibrant, respectively.
This is not so helpful on its own, since our model structure on diagrams uses $q$-cofibrations instead of $h$-cofibrations.
Fortunately, we can define maps of $\cat M_n^\op$-diagrams one level at a time, one cell at a time.
So consider inductively the modified square
\[ \xymatrix{
K \times \left[ \left( \{0\} \times \left( \bigcup_{s \in S'} X_s \right)^{\barsmash i} \right) \cup \left( I \times \left[ \Delta \cup \underset{T \subsetneq S'}\colim \left( \bigcup_{t \in T} X_t \right)^{\barsmash i} \right] \right) \right] \ar[r] \ar[d] & \Sigma^k_{B^i} E^{\barsmash i} \ar[d] \\
K \times I \times \left( \bigcup_{s \in S'} X_s \right)^{\barsmash i} \ar[r] \ar@{-->}[ur] & B^i } \]
where $\Delta \subset (\bigcup_{s \in S'} X_s)^{\barsmash i}$ is the fat diagonal.
From Prop. \ref{diagonalcells} above we know that $(\bigcup_S X_s)^{\barsmash i}$ is built up from its fat diagonal by attaching free $\Sigma_i$-cells, so we can define the lift one free $\Sigma_i$-cell at a time.
Each time, we get an acyclic $h$-cofibration on the left, and the map on the right is an $h$-fibration, so the lift exists.
By construction, it is natural with respect to all the maps in $\cat M_n^\op$.

Now that we have a fibration cube of spaces
\[ T \mapsto \Map\left( \left( \bigcup_{t \in T} X_t \right)^{\barsmash i}, \Sigma^k_{B^i} E^{\barsmash i} \right) \]
we check that the map from the initial vertex into the ordinary limit of the rest is a weak equivalence:
\[ \resizebox{\textwidth}{!}{\xymatrix{
\Map_{(\cat M_n^\op,\{B^i\})}\left( \left( \bigcup_{s \in S} X_s \right)^{\barsmash i}, \Sigma^k_{B^i} E^{\barsmash i} \right) 
\ar[r] & \underset{T \subsetneq S'}\lim\, \Map_{(\cat M_n^\op,\{B^i\})}\left( \left( \bigcup_{t \in T} X_t \right)^{\barsmash i}, \Sigma^k_{B^i} E^{\barsmash i} \right) \ar@{=}[d] \\
& \Map_{(\cat M_n^\op,\{B^i\})}\left( \underset{T \subsetneq S'}\colim \left[\left( \bigcup_{t \in T} X_t \right)^{\barsmash i}\right], \Sigma^k_{B^i} E^{\barsmash i} \right) }} \]
Since $i \leq n < |S|$, every choice of $i$ points in $\bigcup_S X_s$ lies in some $\bigcup_T X_t$ for some proper subset $T$ of $S$.
Therefore this map is a homeomorphism.
\end{proof}

\begin{thm}\label{main}
$P_n F$ is the universal $n$-excisive approximation of $F$.
\end{thm}
\begin{proof}
This follows from Cor. \ref{recognition} above and Thm. \ref{final_third} below.
Together, they tell us that the universal $n$-excisive approximation $P_n F$ exists and is uniquely identified by the property that $P_n F$ is $n$-excisive and $F \ra P_n F$ is an equivalence on the spaces with at most $n$ points.
\end{proof}

\begin{rmk}
Our construction of $P_n F$ is not finitary, but this could be fixed by applying the usual fibrant replacement functor $\Omega^\infty\Sigma^\infty$ to each fiber of the retractive space $\Sigma^k_{B^i} E^{\barsmash i}$, before taking maps in from $X$.
This is a more sophisticated construction, but when $X$ is a finite cell complex, it is equivalent to the simpler one we described in this section.
\end{rmk}

\subsection{The layers}\label{layers}

We would like to identify each layer $D_nF$ of the tower, defined to be the homotopy fiber of
\[ P_n F \ra P_{n-1} F \]
In fact, the natural map $P_n F \ra P_{n-1} F$ is a level fibration of spectra, so $D_n F$ is equivalent to the ordinary fiber.
This fiber consists of all collections of maps that are trivial on $X^{\barsmash i}$ for $i < n$ and that vanish on the fat diagonal of $X^{\barsmash n}$.
Therefore it may be written
\[ D_n F(X) \simeq \Map_{B^n}(X^{\barsmash n}/_{B^n}\Delta,\Sigma^\infty_{B^n} E^{\barsmash n})^{\Sigma_n} \]
where $X^{\barsmash n}/_{B^n}\Delta$ is the fiberwise quotient of $X^{\barsmash n}$ by the fat diagonal $\Delta$ over $B^n$.
More succinctly, it is the pushout of the diagram
\[ \xymatrix{ X^{\barsmash n} \\ \Delta \ar[u] \ar[r] & B^n } \]
In addition, our decoration $(-)^{\Sigma_n}$ means \emph{categorical} fixed points.
So the above spectrum is given at level $k$ by the $\Sigma_n$-equivariant maps 
\[ X^{\barsmash n}/_{B^n}\Delta \ra \Sigma^k_{B^n} E^{\barsmash n} \]
of retractive spaces over $B^n$.

Finally, we specialize to the case where $M$ be a closed manifold, $B = M$, and $X = M \amalg M$.
Consider the following spaces:
\begin{itemize}
\item $\Delta \subset M^n$ is the fat diagonal.
\item $F(M;n) \cong M^n - \Delta$ is the noncompact manifold of ordered $n$-tuples of distinct points in $M$.
\item $\iota: M^n - \Delta \inj M^n$ is the inclusion map.
\item $C(M;n) \cong F(M;n)_{\Sigma_n}$ is the noncompact manifold of unordered $n$-tuples of distinct points in $M$.
\end{itemize}

Then the layers of the above tower can be rewritten
\begin{eqnarray*}
D_n F(M \amalg M) &=& \Gamma_{(M^n,\Delta)}(\Sigma^\infty_{M^n} E^{\barsmash n})^{\Sigma_n} \\
&\simeq & \Gamma_{M^n - \Delta}^c\left(\left. \Sigma^\infty_{M^n} E^{\barsmash n} \right|_{M^n - \Delta}\right)^{\Sigma_n} \\
&\cong & \Gamma_{F(M;n)}^c\left( \Sigma^\infty_{F(M;n)} \iota^* E^{\barsmash n} \right)^{\Sigma_n} \\
&\cong & \Gamma_{C(M;n)}^c(\Sigma^\infty_{C(M;n)} (\iota^* E^{\barsmash n})_{\Sigma_n}) \\
&\simeq & ((\iota^* E^{\barsmash n})_{\Sigma_n})^{-T(C(M;n))}
\end{eqnarray*}
Here $\Gamma_{(A,B)}$ denotes sections over $A$ that vanish on $B$, and $\Gamma_A^c$ denotes sections over $A$ with compact support.
The last step is the application of Poincar\'e duality (see \cite{ms}, \cite{cohen2009umkehr}) to the noncompact manifold $C(M;n)$ with twisted coefficients given by the bundle of spectra $(\iota^* E^{\barsmash n})_{\Sigma_n}$.
Since the manifold in question is noncompact, Poincar\'e duality gives an equivalence between cohomology with compact supports and homology desuspended by the tangent bundle of $C(M;n)$.
The result is the Thom spectrum $((\iota^* E^{\barsmash n})_{\Sigma_n})^{-T(C(M;n))}$.
We will see a few examples of this in the next section.

\section{Examples and calculations}

\begin{ex}
Taking $B = *$ and $E = Y$ for any based space $Y$ gives
\[ \begin{array}{rcl}
F(X) &=& \Sigma^\infty \Map_*(X,Y) \\
\vdots && \vdots  \\
P_n F(X) &=& \Map_{\cat M_n^\op,*}(X^{\sma i}, \Sigma^\infty Y^{\sma i}) \\
\vdots && \vdots  \\
P_2 F(X) &=& \Map_*(X \sma X,\Sigma^\infty Y \sma Y)^{\Sigma_2} \\
P_1 F(X) &=& \Map_*(X,\Sigma^\infty Y) \\
P_0 F(X) &=& *
\end{array} \]
with $n$th layer
\[ D_n F(X) = \Map_*(X^{\sma n}/\Delta,\Sigma^\infty Y^{\sma n})^{\Sigma_n} \]
This coincides with Arone's tower in \cite{arone_snaith}, and therefore converges when the connectivity of $Y$ is at least the dimension of $X$.
It is curious that the Taylor tower in the $Y$ variable should agree with the polynomial interpolation tower in the $X$ variable.
This seems to happen in the more general case $B \neq *$ as well.
\end{ex}

\begin{ex}
If $X = S^1$ and $Y$ is simply connected then the $n$th layer of the tower is
\[ \Map_*(S^n/\Delta,\Sigma^\infty Y^{\sma n})^{\Sigma_n} \cong \Omega^n \Sigma^\infty Y^{\sma n} \]
If $Y = \Sigma Z$ with $Z$ connected, then the $n$th layer is
\[ \Sigma^\infty Z^{\sma n} \]
It is well known that the tower splits in this case (\cite{arone_snaith}, \cite{bodigheimer1987stable}):
\[ \Sigma^\infty \Omega \Sigma Z \simeq \prod_{n=1}^\infty \Sigma^\infty Z^{\sma n} \]
\end{ex}

\begin{ex}
If $X$ is unbased we get the tower
\[ \begin{array}{rcl}
F(X) &=& \Sigma^\infty \Map(X,Y) \\
\vdots && \vdots  \\
P_n F(X) &=& \Map_{\cat M_n^\op}(X^i, \Sigma^\infty Y^{\sma i}) \\
\vdots && \vdots  \\
P_2 F(X) &=& \Map(X \times X,\Sigma^\infty Y \sma Y)^{\Sigma_2} \\
P_1 F(X) &=& \Map(X,\Sigma^\infty Y) \\
P_0 F(X) &=& *
\end{array} \]
If $Y = S^0$ and $X$ is any finite unbased CW complex then the $n$th layer of this tower is
\[ D_n F(X) \simeq \Map(X^n/\Delta,\Sph)^{\Sigma_n} \cong D((X^n/\Delta)_{\Sigma_n}) \]
where $D$ denotes Spanier-Whitehead dual.
If $Y = S^m$ and $m > \dim X$ then the tower converges to $\Sigma^\infty \Map(X,S^m)$, and the $n$th layer is
\[ D_n F(X) \simeq \Map(X^n/\Delta,\Sigma^{mn}\Sph)^{\Sigma_n} \simeq \Sigma^{mn} D((X^n/\Delta)_{\Sigma_n}) \]
\end{ex}

\subsection{Gauge groups and Thom spectra}\label{main_calculation}

Let $B = M$ be a closed connected manifold, and let $\mc P \ra M$ be a $G$-principal bundle.
The \emph{gauge group} $\mc G(\mc P)$ is defined to be the space of automorphisms of $\mc P$ as a principal bundle.
Consider the quotient
\[ \mc P^\ad = \mc P \times_G G^\ad \]
where $G^\ad$ is the group $G$ as a right $G$-space with the conjugation action.
Then we may identify $\mc G(\mc P)$ with the space of sections $\Gamma_M(\mc P^\ad)$.
Taking $E$ to be the ex-fibration $(\mc P^\ad) \amalg M$ and $X$ to be the retractive space $M \amalg M$ gives the tower
\[ \begin{array}{ccccc}
F(M \amalg M) &=& \Sigma^\infty \Gamma_M(\mc P^\ad \amalg M) &\cong & \Sigma^\infty_+ \mc G(\mc P) \\
&&\vdots \\
P_n F(M \amalg M) &=& \Gamma_{(\cat M_n^\op, \{M^i\})}(\Sigma^\infty_{M^i} (\mc P^\ad)^i \amalg M^i) \\
&&\vdots \\
P_2 F(M \amalg M) &=& \Gamma_{M \times M}(\Sigma^\infty_{M \times M} (\mc P^\ad)^2 \amalg M^2)^{\Sigma_2} \\
P_1 F(M \amalg M) &=& \Gamma_M(\Sigma^\infty_M \mc P^\ad \amalg M) &\simeq & (\mc P^\ad)^{-TM} \\
P_0 F(M \amalg M) &=& *
\end{array} \]
To describe the layers, we recall that $C(M;n)$ is the space of configurations of $n$ unordered distinct points in $M$.
We define $\mc C(\mc P^\ad;n)$ as the space of unordered configurations of $n$ points in $\mc P^\ad$, whose images in $C(M;n)$ are distinct.
Then the description of the $n$th layer in section \ref{layers} above becomes the Thom spectrum
\[ D_n F(M \amalg M) \simeq \mc C(\mc P^\ad;n)^{-T(C(M;n))} \]
In summary, this tower relates the stable homotopy type of the gauge group $\mc G(\mc P)$ to Thom spectra of configuration spaces.

If we use orthogonal spectra instead of prespectra, we get a tower of strictly associative ring spectra.
This proves Theorem \ref{intro_principal_tower} from the introduction.

By the Thom isomorphism, the homology of $\mc C(\mc P^\ad; n)^{-TC(M;n)}$ is the same as the homology of the base space $\mc C(\mc P^\ad; n)$, with coefficients twisted by the orientation bundle of $C(M;n)$ pulled back to $\mc C(\mc P^\ad; n)$.
We can calculate this homology using the zig-zag of homotopy pullback squares
\[ \xymatrix{
G^n \ar[d] & G^n \ar[d] & G^n \ar[d] \\
\mc C(\mc P^\ad;n) \ar[d] & \mc F(\mc P^\ad;n) \ar[d] \ar[l]_-{/\Sigma_n} \ar[r] \ar@{}[rd]|<{\Big\lrcorner} \ar@{}[ld]|<{\Big\llcorner} & (\mc P^\ad)^n \ar[d] \\
C(M;n) & F(M;n) \ar[l]_-{/\Sigma_n} \ar[r] & M^n } \]
where $F(M;n) \cong M^n - \Delta$ is ordered configurations of $n$ points in $M$.
Note that the manifold $F(M;n)$ is orientable iff $M$ is orientable, while $C(M;n)$ is orientable iff $M$ is orientable and $\dim M$ is even.

One may also use the scanning map
\[ \begin{array}{c}
C(M;n) \ra \Gamma_M(S^{TM})_n \\
\mc C(\mc P^\ad;n) \ra \Gamma_M(S^{TM} \sma_M (\mc P^\ad \amalg M))_n \end{array} \]
Here the subscript of $n$ denotes sections that are degree $n$ in the appropriate sense.
If $M$ is open, the scanning map gives an isomorphism on integral homology in a stable range \cite{mcduff1975configuration}.
If $M$ is closed, it gives an isomorphism on rational homology in a stable range \cite{church2011homological}.

\subsection{String topology}\label{string_topology}

Now we will finally construct the tower we described in the introduction.
We may start with the tower of section \ref{mapping} and set $B = M$, $E = LM \amalg M$, and $X = M \amalg M$.
Or, we may take the tower from section \ref{main_calculation} and set $G \simeq \Omega M$ and $\mc P \simeq *$, so that $\mc P^\ad \simeq LM$.
Either construction gives the tower
\[ \begin{array}{ccccc}
F(M \amalg M) &=& \Sigma^\infty \Gamma_M(LM \amalg M) &\simeq & \Sigma^\infty_+ \Omega_\id \haut(M) \\
&&\vdots \\
P_n F(M \amalg M) &=& \Gamma_{(\cat M_n^\op, \{M^i\})}(\Sigma^\infty_{M^i} LM^i \amalg M^i) \\
&&\vdots \\
P_2 F(M \amalg M) &=& \Gamma_{M \times M}(\Sigma^\infty_{M \times M} LM^2 \amalg M^2)^{\Sigma_2} \\
P_1 F(M \amalg M) &=& \Gamma_M(\Sigma^\infty_M LM \amalg M) &\simeq & LM^{-TM} \\
P_0 F(M \amalg M) &=& *
\end{array} \]
The $n$th layer is
\[ D_n F(M \amalg M) \simeq \mc C(LM;n)^{-TC(M;n)} \]
As before, $\mc C(LM;n)$ is configurations of $n$ unmarked free loops in $M$ with distinct basepoints.
This proves Theorem \ref{intro_string_tower} from the introduction.

\begin{rmk}
The connectivity of the $n$th layer $\mc C(LM;n)^{-TC(M;n)}$ is negative, and decreases to $-\infty$ as $n \ra \infty$.
Therefore the tower does not converge to $F$.
We may phrase this another way: if the first layer is $k$-connected, then the $n$th layer is approximately $nk$-connected.
This is actually consistent with other results from calculus of functors (cf. \cite{calc3} Thm. 1.13 and \cite{arone_snaith} Thm. 2), the difference here being that $k$ is negative.
In the more general case of an ex-fibration $E \ra M$, the tower does converge when the dimension of $M$ is at most the connectivity of the fibers of $E$.
This will follow from Thm \ref{convergence} below.
\end{rmk}

We conclude this section with a short homology calculation.
We will skip over the first layer $LM^{-TM}$, since it can be calculated using methods from \cite{cohen2004loop}.
Instead, taking $M = S^n$, we use standard Serre spectral sequence arguments to calculate the second layer in rational homology
\[ H_*(\mc C(LS^n;2)^{-TC(S^n;2)};\Q) \]

If $n$ is odd, then $H_q(\mc C(LS^n;2);\Q)$ with twisted coefficients is
\[ \left\{ \begin{array}{cc}
\Q & q = n-1, 2n-2, 2n-1, 3n-2 \\
\Q^2 & q = 3n-3, 4n-4, 4n-3, 5n-4 \\
\vdots & \vdots \\
0 & \textup{otherwise}
\end{array} \right. \]
and if $n$ is even the answer (with untwisted coefficients) is
\[ \left\{ \begin{array}{cc}
\Q & q = n-1, 3n-3, 3n-2, 4n-4, \\
& 5n-4, 6n-6, 6n-5, 8n-7 \\
\Q^2 & q = 5n-5, 7n-7, 7n-6, 8n-8, 9n-9, \\
& 9n-8, 10n-10, 10n-9, 12n-11 \\
\vdots & \vdots \\
0 & \textup{otherwise}
\end{array} \right. \]
To get the homology of the spectrum $\mc C(LS^n;2)^{-TC(S^n;2)}$ we subtract $2n$ from each degree.
This spectrum is a homotopy fiber of a map of rings, so its homology carries an associative multiplication with no unit.
It is easy to check however that most of the products are zero, and when $n$ is odd, all the products are zero.

\section{First construction of $P_n F$}\label{firstconstruction}

We still need to add teeth to Cor. \ref{recognition} by giving a functorial construction of $P_n F$ for general $F$ with the desired properties.
We begin by sketching a description of $P_n F$ in the non-fiberwise case ($B = *$) that the author learned from Greg Arone.
Broadly, the construction in this section is the cellular approach to $P_n F$, whereas our second construction in section \ref{secondconstruction} is the simplicial approach.

Let $F: \mc T_{\fin}^\op \ra \mc T$ be a contravariant homotopy functor from finite based CW complexes to based spaces.
Assume in addition that $F$ is \emph{topological}, meaning it is enriched over unbased spaces.
Then we define $P_n F$ to be the enriched homotopy right Kan extension of $F$, from the category $\mc T_n$ of finite based sets $\il_+ = \{1,\ldots,i\}_+$ with $0 \leq i \leq n$ back to the category of all finite CW complexes $\mc T_{\fin}$.

We can give $P_n F$ more explicitly as follows.
For a fixed finite based space $X$, define two diagrams of unbased spaces over $\mc T_n^\op$:
\begin{align*}
\il_+ &\leadsto X^i = \Map_*(\il_+,X) \\
\il_+ &\leadsto F(\il_+)
\end{align*}
Now consider the space of (unbased) maps between these two diagrams
\[ P_n F(X) \overset{?}= \Map_{\mc T_n^\op}(X^i,F(\il_+)) \]
Note that since $F$ is topological, there is a natural map from $F(X)$ into this space.
Furthermore, this map is a homeomorphism when $X = \il_+$ for $0 \leq i \leq n$, since then the diagram $X^i$ is generated freely by a single point at level $\il_+$, corresponding to the identity map of $\il_+$.
This is good, but we have missed the mark a little bit because this construction is not $n$-excisive in general.

To fix this, we take the \emph{derived} or \emph{homotopically correct} mapping space of diagrams instead.
We do this by fixing a model structure on $\mc T_n^\op$ diagrams in which the weak equivalences are defined objectwise.
Then we replace $\{X^i\}$ by a cofibrant diagram and $\{F(\il_+)\}$ by a fibrant diagram.
The space of maps between these replacements is, by definition, the homotopically correct mapping space.

More concretely, we can fatten up the diagram $\{X^i\}$ to the diagram
\[ \il_+ \leadsto \underset{\jl_+ \in (\mc T_n^\op \downarrow \il_+)}\hocolim\, X^j \]
and leave $\{F(\il_+)\}$ alone.
Then the above conditions are satisfied for the projective model structure defined above in section \ref{cells}.
This standard thickening is sometimes called a \emph{two-sided bar construction} \cite{may1975classifying} \cite{shulman2006homotopy}.

Equivalently, we can leave $\{X^i\}$ alone and fatten up $\{F(\il_+)\}$ to
\[ \il_+ \leadsto \underset{\jl_+ \in (\il_+ \downarrow \mc T_n^\op)}\holim\, F(\jl_+) \]
Then the above conditions would be satisfied for the injective model structure, if it existed.
Note that the two spaces we get in either case are actually homeomorphic:
\[ P_n F(X) = \Map_{\mc T_n^\op}\left[\underset{\jl_+ \in (\mc T_n^\op \downarrow \il_+)}\hocolim\, X^j,F(\il_+)\right] 
\cong \Map_{\mc T_n^\op}\left[ X^j,\underset{\il_+ \in (\jl_+ \downarrow \mc T_n^\op)}\holim\, F(\il_+) \right] \]
Take either of these as our definition of $P_n F(X)$.
The natural map $F(X) \ra P_n F(X)$ can be seen by taking the previous case and observing in addition that there are always natural maps
\[ \hocolim \ra \colim \qquad \text{ or } \qquad \lim \ra \holim \]

Consider the second description
\[ P_n F(X) \cong \Map_{\mc T_n^\op}\left[ X^j,\underset{\il_+ \in (\jl_+ \downarrow \mc T_n^\op)}\holim\, F(\il_+) \right] \]
When $X = \il_+$ and $i \leq n$, we may evaluate our map of diagrams at the ``identity'' point of $(\il_+)^i$, giving a homeomorphism
\[ P_n F(\il_+) \congar \underset{\jl_+ \in (\il_+ \downarrow \mc T_n^\op)}\holim\, F(\jl_+) \]
of spaces under $F(\il_+)$.
But the map
\[ F(\il_+) \simar \underset{\jl_+ \in (\il_+ \downarrow \mc T_n^\op)}\holim\, F(\jl_+) \]
is an equivalence too, and this forces $F(\il_+) \ra P_n F(\il_+)$ to be an equivalence.

Next, consider the first description
\[ P_n F(X) = \Map_{\mc T_n^\op}\left[\underset{\jl_+ \in (\mc T_n^\op \downarrow \il_+)}\hocolim\, X^j,F(\il_+)\right] \]
Remembering that $X$ is a cell complex, the diagram on the left is a cell complex of diagrams.
It has one free cell of dimension $d+m$ at the vertex $\il_+$ for every choice of $d$-cell in $X^j$ and choice of $m$-tuple of composable arrows
\[ \jl_+ = \uline{i_0}_+ \ra \uline{i_1}_+ \ra \ldots \ra \uline{i_m}_+ = \il_+ \]
Therefore the techniques of section \ref{correct} above tell us that $P_n F$ is $n$-excisive.
This completes the proof that $n$-excisive approximations exist for topological functors from based spaces to based spaces.

The assumption that $F$ is topological is not much stronger than the assumption that $F$ is a homotopy functor.
To see this, first define $\Delta_X$ to be the category of simplices $\Delta^p \ra X$.
A map from $\Delta^p \ra X$ to $\Delta^q \ra X$ is a factorization $\Delta^q \inj \Delta^p \ra X$, where $\Delta^q \inj \Delta^p$ is a composition of inclusions of faces.
Then even when $F$ is not topological, each map $\Delta^p_+ \sma Y \ra X$ gives a map
\[ \Delta^p_+ \sma F(X) \simar F(X) \ra F(\Delta^p_+ \sma Y) \]
These assemble into a zig-zag
\[ F(X)
\ra \underset{\Delta_{\Map_*(Y,X)}}\holim\, F(\Delta^p_+ \sma Y)
\overset\sim\la \underset{\Delta_{\Map_*(Y,X)}}\holim\, F(Y) \]
\[ 
\cong \Map(|\Delta_{\Map_*(Y,X)}|,F(Y))
\overset\sim\la \Map(\Map_*(Y,X),F(Y)) \]
assuming $\Map_*(Y,X)$ has the homotopy type of a CW complex.
So we don't quite get a map from $F(X)$ to the far right-hand side, but we get something close enough for the purposes of homotopy theory.
In particular, setting $Y = \il_+$ we get a natural zig-zag
\[ F(X) \ra \ldots \ra \Map(X^i,F(\il_+)) \]
and therefore we get a natural zig-zag from $F(X)$ to $P_n F(X)$, which gives a natural map $F \ra P_n F$ in the homotopy category of functors.
This map is still an equivalence when $X$ has at most $n$ points, because in that case the homotopy limits become ordinary products and we can use the same argument as above.

Once we have the case where $F$ takes based spaces to spaces, we can also easily handle the case when $F$ takes based spaces to spectra.
Simply post-compose $F$ with fibrant replacement of spectra, and work one level at a time.
Of course, the above constructions naturally commute with taking the based loop space $\Omega$, so they pass to a construction on spectra.

If instead $F$ is defined on unbased finite CW complexes $\mc S_\fin$, then we make the same construction, except that we replace $\mc T_n$ with the category of finite unbased sets $\mc S_n$.
Using Prop. \ref{exists} above, we have finished the proof of the following:
\begin{thm}\label{final}
If $F: \mc C^\op \ra \mc D$ is a homotopy functor, where $\mc C = \mc S_{\fin}$ or $\mc T_{\fin}$ and $\mc D = \mc T$ or $\mc Sp$, then there is a universal $n$-excisive approximation $P_n F$, and $F \ra P_n F$ is an equivalence on spaces with at most $n$ points.
\end{thm}

This result is a good first step, but we really want to know that approximations exist for functors defined on the categories $\mc S_{B,\fin}$ and $\mc R_{B,\fin}$ of fiberwise spaces.
We will do this in section \ref{secondconstruction} by switching to a more simplicial construction
\[ P_n F(X) = \underset{\Delta^p \times \il \ra X}\holim\, F(\il \times \Delta^p) \]

\subsection{Higher Brown representability}

Before we move on, we should point out that this first construction is better suited to proving a kind of Brown Representability for homogeneous $n$-excisive functors.
Let $F: \mc C^\op \ra \mc D$ be a homotopy functor as in Thm. \ref{final} above.
Then $F$ is \emph{$n$-reduced} if $P_{n-1}F \simeq *$, or equivalently if $F(X) \simeq *$ whenever $X$ is a space with at most $(n-1)$ points.
Note that
\[ F_n(X) := \hofib(F(X) \ra P_{n-1}F(X)) \]
is always $n$-reduced, and $F_n(\nl)$ is the \emph{cross effect} $\cross_n F(\uline{1},\ldots,\uline{1})$ defined below in section \ref{cross_def}.

We say that $F$ is \emph{homogeneous $n$-excisive} if it is $n$-excisive and $n$-reduced.
So $D_n F(X) = \hofib(P_n F(X) \ra P_{n-1}F(X))$ is always homogeneous $n$-excisive.
Homogeneous $1$-excisive functors are a good notion of space-valued or spectrum-valued cohomology theories.
From numerous sources (e.g. \cite{chorny2007brown}, \cite{cohen2009umkehr}, \cite{ms}) we expect such cohomology theories to be represented by spaces or spectra.

Examining the construction of $P_n F$ in this section, we see that $P_n F(X) \ra P_{n-1} F(X)$ is a Serre fibration when $X$ is a CW complex.
Therefore the ordinary fiber is equivalent to $D_n F$.
This can be rephrased as the following:

\begin{prop}
\begin{itemize}
\item If $F: \mc S_\fin^\op \ra \mc T$ is an $n$-reduced homotopy functor then there is a natural map
\[ F(X) \ra D(X) := \Map_*(X^n/\Delta,F(\nl))^{\Sigma_n} \]
in the homotopy category of functors on $\mc S_\fin^\op$.
If $F$ is homogeneous $n$-excisive then this map is an equivalence.
\item If $F: \mc T_\fin^\op \ra \mc T$ then the same is true for
\[ F(X) \ra D(X) := \Map_*(X^{\sma n}/\Delta,F(\nl_+))^{\Sigma_n} \]
\item Analogous statements hold when the target of $F$ is spectra.
\end{itemize}
\end{prop}

We will now strengthen this to an equivalence of homotopy categories.
Let $G$ be a finite group.
Recall that the usual notion of \emph{$G$-equivalence} of $G$-spaces is an equivariant map $X \ra Y$ which induces equivalences $X^H \ra Y^H$ for all subgroups $H < G$.
We will call an equivariant map $X \ra Y$ a \emph{na\"ive $G$-equivalence} if it is merely an equivalence when we forget the $G$ action.
There are two well-known cofibrantly generated model structures on $G$-spaces, one which gives the $G$-equivalences and one which gives the na\"ive $G$-equivalences.

Examining the behavior of $D(X)$ on the spaces $\il$ or $\il_+$ for $i \leq n$, it is clear that the homotopy type of $D(X)$ is determined by the \emph{nonequivariant} or \emph{na\"ive} homotopy type of $F(\nl)$ or $F(\nl_+)$.
The following is then straightforward:

\begin{prop}
The above construction gives an equivalence between the homotopy category of homogeneous $n$-excisive functors to spaces and the na\"ive homotopy category of $\Sigma_n$-spaces.
A similar statement holds for functors to spectra.
\end{prop}

\section{Properties of homotopy limits}

In order to carry out our second construction of $P_n F$, we need a small collection of facts about homotopy limits.
This section is expository except for Prop. \ref{holim_twisted_split}.

Let $[n]$ denote the totally ordered set $\{0,1,\ldots,n\}$ as a category.
Let $\Delta[p]$ denote the standard $p$-simplex as a simplicial set, and let $\Delta^p = |\Delta[p]|$ denote its geometric realization.
If $\cat I$ is any small category, let $N\cat I$ denote its nerve and let $B\cat I = |N\cat I|$ denote its classifying space.
Recall from \cite{bousfield1987homotopy} that if $A: \cat I \ra \mc T$ is a diagram of based spaces, the homotopy limit is defined as the subspace
\[ \underset{\cat I}\holim\, A \subset \prod_{g: [n] \ra N\cat I} \Map_*(\Delta^n_+,A(g(n))), \]
consisting of all collections of maps that agree in the obvious way with the face and degeneracy maps of the nerve $N\cat I$.
The following is perhaps the most standard result about homotopy limits, and we have already used it several times.
It is included here for completeness.

\begin{prop}
If $A,B: \cat I \ra \mc T$ are two diagrams indexed by $\cat I$, and $A \ra B$ is a natural transformation that on each object $i \in \cat I$ gives a weak equivalence $A(i) \ra B(i)$, then it induces a weak equivalence
\[ \underset{\cat I}\holim\, A \ra \underset{\cat I}\holim\, B \]
\end{prop}

Recall that if $\cat I \overset\alpha\ra \cat J$ is a functor and $A: \cat J \ra \mc T$ is a diagram of spaces, then there is a naturally defined map
\[ \underset{\cat J}\holim\, A \ra \underset{\cat I}\holim\, (A \circ \alpha) \]
The functor $\cat I \overset\alpha\ra \cat J$ is \emph{homotopy initial} (or \emph{homotopy left cofinal}) if for each object $j \in \cat J$ the overcategory $(\alpha \downarrow j)$ has contractible nerve.

\begin{prop}\label{holim_initial}
If $\cat I \overset\alpha\ra \cat J$ is homotopy initial and $A: \cat J \ra \mc T$ is a diagram of spaces, then
\[ \underset{\cat J}\holim\, A \ra \underset{\cat I}\holim\, (A \circ \alpha) \]
is an equivalence.
\end{prop}

The following lemma provides some standard tools for proving that $\alpha$ is homotopy initial.

\begin{lem}
\begin{itemize}
\item Each adjunction of categories induces a homotopy equivalence on the nerves.
\item If $(\alpha \downarrow j)$ is related by a zig-zag of adjunctions to the one-point category $*$, then its nerve is contractible and therefore $\alpha$ is homotopy initial.
\item If $(\alpha \downarrow j)$ has an initial or terminal object then $\alpha$ is homotopy initial.
\item If $\alpha$ is a left adjoint then it is homotopy initial.
\end{itemize}
\end{lem}

The categories indexing our homotopy limits will be categories of simplices in $X$:
\begin{df}\label{deltacat}
\begin{itemize}
\item If $X$ is a space, let $\Delta_X$ denote the category of simplices $\Delta^p \ra X$.
A map from $\Delta^p \ra X$ to $\Delta^q \ra X$ is a factorization
\[ \Delta^q \inj \Delta^p \ra X \]
where $\Delta^q \inj \Delta^p$ is a composition of inclusions of faces.
The classifying space $B\Delta_{X}$ is homeomorphic to the barycentric subdivision of the thick geometric realization of $S_\cdot X$; therefore there is a functorial weak equivalence
\[ B\Delta_X \cong \|S_\cdot X\| \simar |S_\cdot X| \simar X \]
\item If $X_\cdot$ is a simplicial set, there is an obvious analogue of $\Delta_{X_\cdot}$, whose classifying space is homeomorphic to the thick realization of $X_\cdot$:
\[ B\Delta_{X_\cdot} \cong \|X_\cdot\| \simar |X_\cdot| \]
In both cases, the thick and thin realizations are equivalent because every simplicial set is ``good'' in the sense of Segal (\cite{segal1974categories}).
\end{itemize}
\end{df}

Next, we need a fact about iterated homotopy limits.
We recall the colimit version first.
If $F: \cat I \ra \cat{Cat}$ is a small diagram of small categories, the \emph{Grothendieck construction} gives a larger category $\cat I \int F$, whose objects are pairs $(i,x)$ of an object $i \in \cat I$ and an object $x \in F(i)$.
The maps $(i,x) \ra (j,y)$ are arrows $i \overset{f}\ra j$ in $\cat I$, and arrows $F(f)(x) \ra y$ in $F(j)$.
\emph{Thomason's Theorem} tells us that a homotopy colimit of a diagram $A: \cat I \int F \ra \mc T$ is expressed as an iterated homotopy colimit:
\[ \underset{\cat I \int F}\hocolim\, A \simeq \underset{i \in \cat I}\hocolim\,\left( \underset{F(i)}\hocolim\, A \right) \]

To formulate the result for homotopy limits, we again let $F: \cat I \ra \cat{Cat}$ be a small diagram of small categories.
Then the \emph{reverse Grothendieck construction} gives a larger category $\cat I \int^R F$, whose objects are again pairs $(i,x)$ of an object $i \in \cat I$ and an object $x \in F(i)$.
The maps $(i,x) \ra (j,y)$ are arrows $j \overset{f}\ra i$ in $\cat I$, and arrows $x \ra F(f)(y)$ in $F(i)$.
Note that this is related to the original Grothendieck construction in that
\[ \textstyle \cat I \int^R F \cong ( \cat I \int (\op \circ F) )^\op \]

\begin{prop}[Dual of Thomason's Theorem]\label{dual_Thomason}
For a diagram $A: \cat I \int^R F \ra \mc T$, there is a natural weak equivalence
\[ \underset{\cat I \int^R F}\holim\, A \overset\sim\ra \underset{i \in \cat I^\op}\holim\,\left( \underset{F(i)}\holim\, A \right) \]
\end{prop}

We will not give a proof of this since Schlictkrull gives an excellent one in \cite{schlicht}.

In this paper, we will come upon several homotopy limits that are indexed by forwards Grothendieck constructions $\cat I \int F$ instead of reverse ones.
Here we will demonstrate that such a homotopy limit splits, but the result is more complicated.
\begin{df}\label{twisted_arrow}
If $\cat I$ is a small category, the \emph{twisted arrow category} $a\cat I$ has as its objects the arrows $i \ra j$ of $\cat I$.
The morphisms from $i \ra j$ to $k \ra \ell$ are the factorizations of $k \ra \ell$ through $i \ra j$:
\[ \xymatrix{
i \ar[d] & k \ar[d] \ar[l] \\ j \ar[r] & \ell } \]
\end{df}

\begin{prop}\label{holim_twisted_split}
Given a diagram $A: \cat I \int F \ra \mc T$ there is a natural weak equivalence
\[ \underset{\cat I \int F}\holim\, A \overset\sim\ra \underset{(i \overset{f}\rightarrow j) \in a\cat I}\holim\,\left( \underset{F(i)}\holim\, A \circ F(f) \right) \]
\end{prop}

\begin{rmk}
This proposition is motivated by a result of Dwyer and Kan on function complexes \cite{dwyerkanfunction}.
Roughly, the left-hand side is the space of maps between two diagrams indexed by $\cat I$.
The first diagram sends $i$ to the nerve of $F(i)$, while the other sends $i$ to $A(i)$.
Mapping spaces of this form, if they are ``homotopically correct,'' are equivalent to a homotopy limit of mapping spaces $\Map(NF(i), A(j))$ over the twisted arrow category $a\cat I$; this is roughly what we get on the right-hand side.
\end{rmk}

\begin{proof}
Recall that we already have a functor $F: \cat I \ra \cat{Cat}$.
Define another functor $(a\cat I)^\op \ra \cat{Cat}$ by taking $i \ra j$ to $F(i)$, and call this functor $F$ by abuse of notation.
Then we can build the reverse Grothendieck construction $(a\cat I)^\op \int^R F$.

The desired weak equivalence is the composite
\[ \underset{\cat I \int F}\holim\, A \simar
\underset{(a\cat I)^\op \int^R F}\holim\, A \circ \alpha \simar
\underset{(i \overset{f}\rightarrow j) \in a\cat I}\holim\,\left( \underset{F(i)}\holim\, A \circ F(f) \right) \]
The second map is a weak equivalence by the dual of Thomason's theorem, stated above.
The first map is induced by pullback along a functor
\[ (a\cat I)^\op \int^R F \overset\alpha\ra \cat I \int F \]
and it suffices to show that this functor is homotopy initial.
Specifically, $\alpha$ does the following to objects and morphisms:
\[ \xymatrix{
(i \ar[r]^-{f} & j, \ar[dd]^-{h} & x \in F(i)) \ar[d]^-\varphi \ar@{~>}[rr]^-\alpha && (j, \ar[dd]^-{h} & F(f)(x) \in F(j) \ar@{~>}[d]^-{F(h)} ) \\
&& F(g)(x') \in F(i) && & F(hf)(x) \in F(j') \ar[d]^-{F(hf)\varphi} \\
(i' \ar[r]^-{f'} \ar[uu]_-{g} & j', & x' \in F(i')) \ar@{~>}[u]_-{F(g)} \ar@{~>}[rr]^-\alpha && (j', & F(f')(x') \in F(j')) } \]
Fix an object $(\ell, z \in F(\ell))$ in the target category $I \int F$.
We'll show that the overcategory $(\alpha \downarrow (\ell,z))$ is contractible.
A typical map between objects of this overcategory is given by the data
\[ \xymatrix{
i \ar[r]^-{f} & j \ar[r]^-{p} \ar[dd]^-{h} & \ell, \ar@{=}[dd] & x \in F(i) \ar[d]^-\varphi \ar@{~>}[r]^-{F(pf)} & F(pf)(x) \ar[dd]^-{F(pf)\varphi} \ar[r]^-\sigma & z \ar@{=}[dd] \\
&&& F(g)(x') \in F(i)  \ar@{~>}[rd]^-{F(pf)} & \\
i' \ar[r]^-{f'} \ar[uu]_-{g} & j' \ar[r]^-{p'} & \ell, & x' \in F(i') \ar@{~>}[u]_-{F(g)} \ar@{~>}[r]^-{F(p'f')} & F(p'f')(x') \ar[r]^-{\sigma'} & z } \]
where everything commutes.
Let $\cat J$ be the subcategory of $(\alpha \downarrow (\ell,z))$ consisting of objects for which $j = \ell$ and $p$ is the identity.
Then there is a projection $P: (\alpha \downarrow (\ell,z)) \ra \cat J$ which is left adjoint to the inclusion $I: \cat J \ra (\alpha \downarrow (\ell,z))$.
We can exhibit $P$ and the natural transformation from the identity to $I \circ P$ in the following diagram:
\[ \xymatrix{
i \ar[r]^-{f} & j \ar[r]^-{p} \ar[dd]^-{p} & \ell, \ar@{=}[dd] & x \in F(i) \ar@{=}[d] \ar@{~>}[r]^-{F(pf)} & F(pf)(x) \ar@{=}[dd] \ar[r]^-\sigma & z \ar@{=}[dd] \\
&&& x \in F(i) \ar@{~>}[rd]^-{F(pf)} & \\
i \ar[r]^-{pf} \ar@{=}[uu] & \ell \ar@{=}[r] & \ell, & x \in F(i) \ar@{~>}[u]_-{\id} \ar@{~>}[r]^-{F(pf)} & F(pf)(x) \ar[r]^-{\sigma} & z } \]
To check the adjunction, it suffices to check that a map from any object of $(\alpha\downarrow(\ell,z))$ into an object of $\cat J$ factors uniquely through this projection.
Once this is checked, the next step is to show that $\cat J$ has an initial subcategory $\cat K$.
A typical object of $\cat K$ is given in the first row below.
\[ \xymatrix{
\ell \ar@{=}[r] & \ell \ar@{=}[r] \ar@{=}[dd] & \ell, \ar@{=}[dd] & F(f)(x) \in F(\ell) \ar@{=}[d] \ar@{~>}[r]^-{\id} & F(f)(x) \ar@{=}[dd] \ar[r]^-\sigma & z \ar@{=}[dd] \\
&&& F(f)(x) \in F(\ell) \ar@{~>}[rd]^-{\id} & \\
i \ar[r]^-{f} \ar[uu]_-{f} & \ell \ar@{=}[r] & \ell, & x \in F(i) \ar@{~>}[u]_-{F(f)} \ar@{~>}[r]^-{F(f)} & F(f)(x) \ar[r]^-{\sigma} & z } \]
The rest of the diagram justifies the claim that $\cat K$ is initial.
Finally, $\cat K$ is isomorphic to the category of objects over $z$ in $F(\ell)$, which has terminal object $z$.
We have completed a zig-zag of adjunctions between $(\alpha \downarrow (\ell,z))$ and $*$, so $(\alpha \downarrow (\ell,z))$ is contractible.
Therefore $\alpha$ is homotopy initial and the equivalence is complete.

The equivalence is clearly natural in $A$, but it is also natural in $F$ in the following sense.
A map of diagrams of categories $F \overset{\eta}\ra G$ gives a map $\cat I \int F \overset{\cat I \int \eta}\ra \cat I \int G$, so a diagram $A: \cat I \int G \ra \mc T$ can be pulled back to $\cat I \int F$.
Our equivalence then fits into a commuting square:
\[ \xymatrix{
\underset{\cat I \int F}\holim\, (\cat I \int \eta)^*A \ar[r] &
\underset{(i \overset{f}\rightarrow j) \in a\cat I}\holim\,\left( \underset{F(i)}\holim\, ((\cat I \int \eta)^*A) \circ F(f) \right) \\
\underset{\cat I \int G}\holim\, A \ar[r] \ar[u] &
\underset{(i \overset{f}\rightarrow j) \in a\cat G}\holim\,\left( \underset{G(i)}\holim\, A \circ G(f) \right) \ar[u] } \]
\end{proof}

Lastly, we want a result on diagrams $A: \cat J \ra \mc T$ for which every arrow $i \ra j$ induces a weak equivalence $A(i) \ra A(j)$.
Call such a diagram \emph{almost constant}.
Of course, if $A$ is a constant diagram sending everything to the space $X$, then its homotopy limit is
\[ \underset{\cat J}\holim\, A = \Map(B \cat J,X) \]
where $B\cat J = |N\cat J|$ is the classifying space of $\cat J$.
If $A$ is instead almost constant, then we get (see \cite{cohen2009umkehr}, \cite{dwyerBG})

\begin{prop}
If $A: \cat J \ra \mc T$ is almost constant, then there is a fibration $E_A \ra B\cat J$ and a natural weak equivalence
\[ \underset{\cat J}\holim\, A \simeq \Gamma_{B\cat J}(E_A) \]
Moreover, if $\cat I \overset{\alpha}\ra \cat J$ is a functor then there is a homotopy pullback square
\[ \xymatrix{
E_{A \circ \alpha} \ar[r] \ar[d] & E_A \ar[d] \\
B\cat I \ar[r] & B\cat J } \]
\end{prop}

\begin{cor}\label{holim_ac}
If $A: \cat J \ra \mc T$ is almost constant, and $\cat I \overset{\alpha}\ra \cat J$ induces a weak equivalence $B\cat I \ra B\cat J$, then the natural map
\[ \underset{\cat J}\holim\, A \ra \underset{\cat I}\holim\, (A \circ \alpha) \]
is a weak equivalence.
\end{cor}

\section{Second construction of $P_n F$: the higher coassembly map}\label{secondconstruction}

Here we will describe how to construct $P_n F(X)$ as a homotopy limit
\[ P_n F(X) = \underset{\Delta^p \times \il \ra X}\holim\, F(\il \times \Delta^p) \]
When $n = 1$ and $F$ is reduced, this construction is essentially the same as the \emph{coassembly map} described in \cite{cohen2009umkehr}.
The coassembly map is formally dual to the \emph{assembly map} (\cite{weiss1993assembly}) often found in treatments of algebraic K-theory.

We will prove that our construction of $P_n F$ satisfies four properties:
\begin{enumerate}
\item $P_n F$ is a homotopy functor.
\item $P_n F$ is $n$-excisive.
\item If $X$ is a CW complex then
\[ P_n F(X) \ra \underset{X' \subset X \textup{ finite complex}}\holim\, P_n F(X') \]
is an equivalence.
\item $F \ra P_n F$ is an equivalence on $\mc R_{B,n}^\op$ or $\mc S_{B,n}^\op$.
\end{enumerate}
For functors on finite CW complexes, conditions (1), (2), and (4) are enough to imply $P_n F$ is the universal $n$-excisive approximation of $F$.
Condition (3) is a bit weaker than the standard condition that filtered homotopy colimits go to homotopy limits; it is here because the technology we need for (2) happens to make (3) easy.

There are 8 different setups we might consider, based on whether $B$ is a point or not, the spaces over $B$ are fiberwise based (retractive) or unbased, and $F$ goes into spaces or spectra.
We will first handle all cases where the spaces over $B$ are unbased.
Then we'll handle all cases where $B = *$ and the spaces over $B$ are based.
Together this gives an extension and a second proof of Theorem \ref{final} above:
\begin{thm}\label{final_second}
If $F: \mc C \ra \mc D$ is a homotopy functor, where $\mc C = \mc S_{B,\fin}$ or $\mc T_{\fin}$ and $\mc D = \mc T$ or $\mc Sp$, then there is a universal $n$-excisive approximation $P_n F$, and $F \ra P_n F$ is an equivalence on spaces with at most $n$ points.
\end{thm}

Finally, in section \ref{hack} below we will do the case of functors from retractive spaces over $B$ to spectra, since the methods we have developed here seem to break down in the case of retractive spaces over $B$ to spaces.

\subsection{$P_n F$ for unbased spaces over $B$}\label{second_construction_unbased}

Let $\cat C_{B,n}$ denote a subcategory of simplicial sets over $S_\cdot B$ consisting of objects of the form
\[ \il \times \Delta[p], \qquad p \geq 0 \text{ and } 0 \leq i \leq n. \]
Specifically, we take one such object for each choice of $p$ and $i$, \emph{and} each choice of map of simplicial sets $\il \times \Delta[p] \ra S_\cdot B$.
We do \emph{not} take the full subcategory on these objects.
Each map
\[ \jl \times \Delta[q] \ra \il \times \Delta[p] \]
must be a product of an injective simplicial map $\Delta[q] \ra \Delta[p]$ and a map of finite sets $\jl \ra \il$.
Intuitively, $\cat C_{B,n}$ is a simplicial fattening of $\mc S_{B,n}$.

Now let $F$ be any contravariant homotopy functor from unbased spaces over $B$ to spaces or spectra.
If $F$ is a functor to spectra, compose it with fibrant replacement.
This gives an equivalent functor that takes weak equivalences of spaces to level equivalences of spectra, and we can argue one level at a time.
So now without loss of generality, $F$ is a homotopy functor to based spaces.

If $X_\cdot$ is a simplicial set over $S_\cdot B$, define
\[ P_n F(X_\cdot) = \underset{(\cat C_{B,n} \downarrow X_\cdot)^\op}\holim\, F(\il \times \Delta^p) \]
Abusing notation, define $P_n F$ on spaces as the composite
\[ \mc S_B \overset{S_\cdot}\ra \cat{sSet}/{S_\cdot B} \overset{P_n F}\ra \mc T \]
or more explicitly,
\[ P_n F(X) = \underset{(\cat C_{B,n} \downarrow S_\cdot X)^\op}\holim\, F(\il \times \Delta^p) \]
The natural transformation $F \overset{p_n}\ra P_n F$ is then induced by a collection of maps $F(X) \ra F(\il \times \Delta^p)$ for each map $\il \times \Delta^p \ra X$.

When $X = \il$ and $i \leq n$, we may take the subcategory $\cat I \subseteq (\cat C_{B,n} \downarrow S_\cdot X)^\op$ of objects of the form $\il \times \Delta[p] \ra S_\cdot \il$, which on connected components give the identity map of $\il$.
Then $\cat I$ is a homotopy initial subcategory, because for a fixed object $\jl \times \Delta[q] \ra S_\cdot \il$, the overcategory $(\cat I \downarrow \jl \times \Delta[q])$ has terminal object $\il \times \Delta[q]$ and is therefore contractible.
On the other hand, $\cat I$ is itself equivalent to the category of simplices $\Delta_{*}$, whose classifying space is contractible.
Combining Prop. \ref{holim_initial} and \ref{holim_ac}, we get that the above homotopy limit is obtained by evaluating at $\il \times \Delta[0]$.
This proves property (4), that $F \ra P_n F$ is an equivalence on $\mc S_{B,n}^\op$.

Next we'll tackle property (1), that $P_n F$ is a homotopy functor.
Let $\mc S_n = \mc S_{*,n}$ be the category of finite unbased sets $\uline{0},\ldots,\uline{n}$ and all maps between them.
Notice that we can define a functor $\Delta: \mc S_n^\op \ra \cat{Cat}$ taking $\il$ to $\Delta_{X_\cdot^i}$.
Each map $\il \la \jl$ goes to the functor $\Delta_{X_\cdot^i} \ra \Delta_{X_\cdot^j}$ arising from the map $X_\cdot^i \ra X_\cdot^j$, whose definition is obvious once we observe that $X^i \cong \Map(\il,X)$.
Now take the forwards Grothendieck construction $\mc S_n^\op \int \Delta$.
This is a category whose objects are elements $X_p^i$ and whose morphisms $X_p^i \ra X_q^j$ are compositions of maps $X_\cdot^i \ra X_\cdot^j$ from above and maps $X_p^j \ra X_q^j$ which are compositions of face maps.
Equivalently, the objects can be described as maps
\[ \Delta[p] \times \il \ra X_\cdot \]
and the morphisms are factorizations
\[ \xymatrix{
\Delta[p] \times \il \ar[r] & X_\cdot \ar@{=}[d] \\
\Delta[q] \times \jl \ar[r] \ar[u] & X_\cdot } \]
in which the vertical map is a product of $\jl \ra \il$ and some injective simplicial map $\Delta[q] \ra \Delta[p]$.
This is clearly the same category as $(\cat C_{B,n}^\op \downarrow X_\cdot)^\op$, so we have a new way to write our definition of $P_n F(X_\cdot)$:
\[ P_n F(X_\cdot) = \underset{\mc S_n^\op \int \Delta}\holim\, F(\il \times \Delta^p) \]

Now Prop. \ref{holim_twisted_split} gives the following:
\[ \underset{\mc S_n \int \Delta_{X^i}}\holim\, F(\Delta^p \times \il) \simeq \underset{(\il \la \jl) \in a\mc S_n^\op}\holim\, \left( \underset{\Delta_{X_\cdot^i}}\holim\, F(\jl \times \Delta^p) \right) \]
The term inside the parentheses defines a homotopy functor in $X_\cdot$ by Prop. \ref{holim_ac}.
The homotopy limit of these is also a homotopy functor, and using the naturality statement in Prop. \ref{holim_twisted_split} we conclude that $P_n F(X_\cdot)$ is a homotopy functor.
In fact, we have proven something stronger than (1), that $P_n F$ actually takes weak equivalences of simplicial sets to weak equivalences.

Now we can prove (2).
From \cite{calc2}, each strongly co-Cartesian cube of spaces over $B$ is weakly equivalent to a pushout cube formed by a cofibrant space $A$ and an $(n+1)$-tuple of spaces $X_0$, $\ldots$, $X_n$ over $B$, each with a cofibration $A \ra X_i$.
Applying singular simplices $S_\cdot$, we get a cube of simplicial sets
\[ T \leadsto S_\cdot \left(\bigcup_{s \in T} X_s\right) \]
where the $\bigcup$ is shorthand for pushout of spaces along $A$.
By easy induction, this cube is equivalent to the pushout cube of simplicial sets
\[ T \leadsto \bigcup_{s \in T} S_\cdot X_s \]
where the $\bigcup$ is shorthand for pushout of simplicial sets along $S_\cdot A$.
Since $P_n F$ is a homotopy functor on simplicial sets, applying $P_n F$ to both cubes gives equivalent results.
Therefore it suffices to show that $P_n F$ takes a pushout cube of simplicial sets to a Cartesian cube of spaces.

So let $S$ by any set with cardinality strictly larger than $n$, let $A \in \cat{sSet}$ be a simplicial set, and for each element $s \in S$, let $X_s \in \cat{sSet}$ be a simplicial set containing $A$.
Then there is a pushout cube which assigns each subset $T \subset S$ to the simplicial set $\bigcup_{t \in T} X_t$, which is shorthand for the pushout of the $X_t$ along $A$.
We want to show that $P_n F$ takes this to a Cartesian cube.
In other words, the natural map
\[ \underset{\Delta[p] \times \il \rightarrow \bigcup_S X_s}\holim\, F(\il \times \Delta^p) \ra \underset{(T \subsetneq S)^\op}\holim\, \left(\underset{\Delta[p] \times \il \rightarrow \bigcup_T X_s}\holim\, F(\il \times \Delta^p)\right) \]
is an equivalence.
Using dual Thomason, we rewrite the right-hand side as
\[ \underset{(T,\Delta[p] \times \il \rightarrow \bigcup_T X_s)}\holim\, F(\il \times \Delta^p) \]
where each object of the indexing category is a proper subset $T \subsetneq S$, integers $p \geq 0$ and $0 \leq i \leq n$, and a map $\Delta[p] \times \il \rightarrow \bigcup_T X_s$.
A map between two objects looks like
\[ \xymatrix{
T, & \il \times \Delta[p] \ar[r] & \bigcup_T X_s \\
U, \ar[u]^*\txt{subset} & \jl \times \Delta[q] \ar[r] \ar[u] & \bigcup_U X_s \ar[u] } \]
This category maps forward into $\mc S_n^\op \int \Delta_{(\bigcup_S X_s)^i}$, in which a map between two objects is given by the data
\[ \xymatrix{
\il \times \Delta[p] \ar[r] & \bigcup_S X_s \\
\jl \times \Delta[q] \ar[r] \ar[u] & \bigcup_S X_s \ar[u] } \]
This functor $\alpha$ forgets the data of $T$ and includes $\bigcup_T X_s^i$ into $\bigcup_S X_s^i$.
The natural map of homotopy limits
\[ \underset{\Delta[p] \times \il \rightarrow \bigcup_S X_s}\holim\, F(\il \times \Delta^p) \ra \underset{(T,\Delta[p] \times \il \rightarrow \bigcup_T X_s)}\holim\, F(\il \times \Delta^p) \]
is induced by a pullback of diagrams along $\alpha$, so we just have to show that $\alpha$ is homotopy initial.
Given an object $\jl \times \Delta[q] \overset\varphi\ra \bigcup_S X_s$ in the target category, the overcategory $(\alpha \downarrow \varphi)$ has as its objects the factorizations of $\varphi$
\[ \jl \times \Delta[q] \ra \il \times \Delta[p] \ra \bigcup_T X_s \ra \bigcup_S X_s \]
where $T \subsetneq S$ must be a \emph{proper} subset of $S$.

Let us give a terminal object for this overcategory.
Since we are working with simplicial sets instead of spaces, each $q$-simplex lands \emph{inside} one of the sets $X_s$ in the pushout.
Therefore there is a smallest subset $T \subset S$ such that $\jl \times \Delta[q] \overset\varphi\ra \bigcup_S X_s$ lands inside $\bigcup_T X_s$, and since $j \leq n < |S|$, this subset is proper.
This gives a terminal object for the overcategory $(\alpha \downarrow \varphi)$, so it is contractible, which finishes (2).

Finally we check (3).
Let $X$ be a CW complex.
We want to show that the natural map
\[ \underset{\Delta[p] \times \il \rightarrow S_\cdot X}\holim\, F(\il \times \Delta^p) \ra \underset{(\textup{finite}\, X' \subset X)^\op}\holim\, \left(\underset{\Delta[p] \times \il \rightarrow S_\cdot X'}\holim\, F(\il \times \Delta^p)\right) \]
is an equivalence.
Using dual Thomason, we rewrite the right-hand side as
\[ \underset{(\textup{finite}\, X' \subset X,\Delta[p] \times \il \rightarrow S_\cdot X')}\holim\, F(\il \times \Delta^p) \]
where each object of the indexing category is a finite subcomplex $X' \subset X$, integers $p \geq 0$ and $0 \leq i \leq n$, and a map $\Delta[p] \times \il \rightarrow S_\cdot X'$.
A map between two objects looks like
\[ \xymatrix{
X', & \il \times \Delta[p] \ar[r] & S_\cdot X' \\
X'', \ar[u]^*\txt{inclusion} & \jl \times \Delta[q] \ar[r] \ar[u] & S_\cdot X'' \ar[u] } \]
This category maps forward into $\mc S_n^\op \int \Delta_{(S_\cdot X)^i}$, in which a map between two objects looks like
\[ \xymatrix{
\il \times \Delta[p] \ar[r] & S_\cdot X \\
\jl \times \Delta[q] \ar[r] \ar[u] & S_\cdot X \ar[u] } \]
This functor $\alpha$ forgets the data of $X'$ and includes $X'$ into $X$.
The natural map of homotopy limits defined above is again induced by a pullback of diagrams along $\alpha$, so we just have to show that $\alpha$ is homotopy initial.
Given an object $\jl \times \Delta[q] \overset\varphi\ra S_\cdot X$ in the target category, the overcategory $(\alpha \downarrow \varphi)$ has as its objects the factorizations of $\varphi$
\[ \jl \times \Delta[q] \ra \il \times \Delta[p] \ra S_\cdot X' \ra S_\cdot X \]
where $X' \subset X$ must be a finite subcomplex.
But of course each $q$-simplex lands inside a unique smallest subcomplex; taking the union over all $\jl$ gives a smallest finite subcomplex containing the image of $\Delta^q \times \jl$.
This gives a terminal object for the overcategory $(\alpha \downarrow \varphi)$, so it is contractible and we are done proving (3).

\subsection{$P_n F$ for based spaces}

The argument mimics the one above, so we will only point out what is different.
The category $\cat C_n$ becomes a subcategory of based simplicial sets consisting of objects of the form
\[ (\il \times \Delta[p])_+, \qquad p \geq 0 \text{ and } 0 \leq i \leq n. \]
with one such object for each choice of $p$ and $i$.
Each map
\[ (\jl \times \Delta[q])_+ \ra (\il \times \Delta[p])_+ \]
is a choice of injective simplicial map $\Delta[q] \ra \Delta[p]$ and map of finite based sets $\jl_+ \ra \il_+$.
Intuitively, $\cat C_n$ is a simplicial fattening of $\mc T_n$.
If $X_\cdot$ is a based simplicial set, define
\[ P_n F(X_\cdot) = \underset{(\cat C_n \downarrow X_\cdot)^\op}\holim\, F((\il \times \Delta^p)_+) \]
Abusing notation, define $P_n F$ on spaces as the composite
\[ \mc T \overset{S_\cdot}\ra \cat{sSet}_* \overset{P_n F}\ra \mc Sp \]
The category $\mc S_n$ of finite sets is replaced by the category $\mc T_n$ of finite based sets.
As before, there is a functor $\Delta: \mc T_n^\op \ra \cat{Cat}$ taking $\il_+$ to $\Delta_{X_\cdot^i}$, and we can rewrite $P_n F(X_\cdot)$ as
\[ P_n F(X_\cdot) = \underset{\mc T_n^\op \int \Delta}\holim\, F((\il \times \Delta^p)_+) \]
To show that $P_n F$ is homotopy invariant we rewrite it as
\[ \underset{\mc T_n^\op \int \Delta}\holim\, F(\Delta^p \times \il) \simeq \underset{(\il_+ \la \jl_+) \in a\mc T_n^\op}\holim\, \left( \underset{\Delta_{X_\cdot^i}}\holim\, F((\jl \times \Delta^p)_+) \right) \]
which proves (1). The proof of (2) and (3) is the same as in the unbased case.

\subsection{Difficulties with retractive spaces over $B$}\label{difficult}

The above proof does not work when generalized to retractive spaces over $B$.
We may define $\mc U_{B,n}$ as the subcategory of spaces \emph{under} B consisting of spaces of the form
\[ \il \amalg B, \quad 0 \leq i \leq n \]
So a map $\il \amalg B \ra \jl \amalg B$ must act as the identity on $B$, but the points in $\il$ may map into $\jl$ or anywhere into $B$.
Then we may define $\mc U_{B,n}^\op \int \Delta$, and then define $P_n F$ as a homotopy limit over this category.
The proof of (1), (2) and (3) is then straightforward.
However, our argument for (4) does not work because there aren't enough maps in $\mc U_{B,n}^\op \int \Delta$ to make our desired object initial.

Examining this shortcoming, it seems one must enrich $\mc U_{B,n}^\op$ and use an enriched version of the above theorems on homotopy limits.
This is not entirely straightforward, since in order to define $P_n F$ here, one must deal with the concept of a ``diagram'' that is indexed not by a simplicially enriched category but by a simplicial object in $\cat{Cat}$.
Instead of doing this, we will handle the case of $F: \mc R_{B,\fin}^\op \ra \mc Sp$ in section \ref{hack} using splitting theorems that only hold for functors into spectra.

\section{Spectra and cross effects}\label{cross_def}

From here onwards we will only consider functors from retractive spaces over $B$ to spectra.
In this section the word \emph{spectra} will refer to prespectra, though the arguments will also work for orthogonal spectra as defined in \cite{mmss}.
Let $\fib$ denote homotopy fiber and $\cofib$ denote (reduced) homotopy cofiber.
For spaces, these have the usual definition
\[ \begin{array}{c}
\fib(X \ra Y) = X \times_Y \Map_*(I,Y) \\
\cofib(X \ra Y) = (X \sma I) \cup_X Y
\end{array} \]
and for spectra these definitions are applied to each level separately.
We recall the following standard facts about spectra and splitting.

\begin{prop}
Suppose that $X$, $X'$, and $Y$ are spectra with maps
\[ X \overset{i}\ra Y \overset{p}\ra X' \]
such that $p \circ i$ is an equivalence.
Then there are natural equivalences of spectra
\[ X \vee \fib(p) \simar Y \simar X \times \cofib(i) \]
which also yield an equivalence $\fib(p) \simar \cofib(i)$.
\end{prop}

\begin{cor}\label{retractsplit}
If $X$ is a retract of $Y$ then $Y \simeq X \vee Z$ where
\[ Z \simeq \fib(Y \ra X) \simeq \cofib(X \ra Y) \]
\end{cor}

\begin{cor}\label{smallsplit}
If $X$ is a well-based space then there is a natural equivalence
\[ \Sigma^\infty(X_+) \simeq \Sigma^\infty(X \vee S^0) \]
\end{cor}

\begin{cor}\label{dumb}
If $\mc R_B^{(\op)} \overset{F}\ra \mc Sp$ is any covariant or contravariant functor then there is a splitting of functors
\[ F(X) \simeq F(B) \times \overline{F}(X) \]
where $\overline{F}(X)$ can be defined as the fiber of $F(X) \ra F(B)$ or the cofiber of $F(B) \ra F(X)$.
This also holds if $F$ is only defined on finite CW complexes.
\end{cor}

We want a slight generalization of these results to $n$-dimensional cubes of retracts.
First recall the higher-order versions of homotopy fiber and homotopy cofiber from \cite{calc2}.
If $F$ is a $n$-cube of spectra then we can think of it as a map between two $(n-1)$-cubes.
The \emph{total homotopy fiber} $\tfib(F)$ is inductively defined as the homotopy fiber of the map between the total homotopy fibers of these two $(n-1)$-cubes.
For a 0-cube consisting of the space $X$, we define the total fiber to be $X$.
Therefore the total fiber of a 1-cube $X \ra Y$ is $\fib(X \ra Y)$.

The \emph{total homotopy cofiber} $\tcofib(F)$ has a similar inductive definition.
Recall from \cite{calc2} that a cube is Cartesian iff its total fiber is weakly contractible, and co-Cartesian iff its total cofiber is weakly contractible.
From this it quickly follows that a cube of spectra is Cartesian iff it is co-Cartesian.

If $F$ is a functor $\mc R_B^\op \ra \mc Sp$, the \emph{$n$th cross effect} $\cross_n F(X_1,\ldots,X_n)$ is defined as in \cite{calc3} to be the total fiber of the cube
\[ S \subset \nl \quad \leadsto \quad F\left(\bigcup_{i \in S} X_i\right) \]
whose maps come from inclusions of subsets of $\nl$.
Here the big union denotes pushout along $B$; one can think of it as a fiberwise wedge sum.
Since $F$ is contravariant, the initial vertex of this cube corresponds to the full subset $S = \nl$.
Note that there is a natural map
\[ \cross_n F(X_1,\ldots,X_n) \overset{i_n}\ra F\left(\bigcup_{i \in \nl} X_i\right) \]

Similarly, the \emph{$n$th co-cross effect} $\cocross_n F(X_1,\ldots,X_n)$ is defined as in \cite{mccarthy2001dual} and \cite{ching2010chain} to be the total cofiber of the cube with the same vertices
\[ S \subset \nl \quad \leadsto \quad F\left(\bigcup_{i \in S} X_i\right) \]
where the maps come from the opposites of inclusions of subsets of $\nl$.
Each inclusion $S \subseteq T$ results in a collapsing map
\[ \bigcup_{i \in S} X_i \la \bigcup_{i \in T} X_i \]
which becomes
\[ F\left(\bigcup_{i \in S} X_i\right) \ra F\left(\bigcup_{i \in T} X_i\right) \]
Note that the final vertex of this cube corresponds to $S = \nl$, so there is a natural map
\[ F\left(\bigcup_{i \in \nl} X_i\right) \overset{p_n}\ra \cocross_n F(X_1,\ldots,X_n) \]
It is known that the cross effect and co-cross effect are equivalent, when $F$ is a functor from spectra to spectra (\cite{ching2010chain}, Lemma 2.2).
A similar argument gives the following.

\begin{prop}\label{splitting}
If $\mc R_B^\op \overset{F}\ra \mc Sp$ is any contravariant functor, then the composite
\[ \cross_n F(X_1,\ldots,X_n) \overset{i_n}\ra F\left(\bigcup_{i \in \nl} X_i\right) \overset{p_n}\ra \cocross_n F(X_1,\ldots,X_n) \]
is an equivalence.
Furthermore, $F(\bigcup X_i)$ splits into a sum of cross-effects:
\begin{eqnarray*}
F\left(\bigcup_{i \in \nl} X_i\right) &\simeq & \prod_{S \subseteq \nl} \cocross_{|S|} F(X_s : s \in S) \\
&\simeq & \prod_{S \subseteq \nl} \cross_{|S|} F(X_s : s \in S) \\
&\simeq & \bigvee_{S \subseteq \nl} \cross_{|S|} F(X_s : s \in S)
\end{eqnarray*}
The analogous result also holds for covariant functors, and for functors defined only on finite CW complexes.
\end{prop}

\begin{rmk}
This does \emph{not} assume that $F$ is a homotopy functor.
\end{rmk}

\begin{proof}
The argument is by induction on $n$.
We form the maps
\[ \bigvee_{S \subsetneq \nl} \cross_{|S|} F(X_s : s \in S) \ra F\left(\bigcup_{i \in \nl} X_i\right) \ra \prod_{S \subsetneq \nl} \cocross_{|S|} F(X_s : s \in S) \]
and observe that the composite is an equivalence.
Therefore the middle contains either of the outside terms as a summand.
We use the alternate definitions of $\tfib$ and $\tcofib$ found in \cite{calc2} to identify the leftover summand with $\cross_{|S|} F$ and $\cocross_{|S|} F$, which proves that they are equivalent and that $F$ splits into a sum of cross effects.
\end{proof}

This generalizes the following well known result: (cf. \cite{browder1969kervaire}, \cite{cohen1980stable})

\begin{cor}[Binomial Theorem for Suspension Spectra]\label{binomial}
If $X$ and $Y$ are well-based spaces then the obvious projection maps yield a splitting
\[ \Sigma^\infty (X \times Y) \simar \Sigma^\infty (X \sma Y) \times \Sigma^\infty X \times \Sigma^\infty Y \]
If $X_1$, $\ldots$, $X_n$ are well-based spaces then we get a more general splitting
\[ \Sigma^\infty \prod_{i=1}^n X_i \simar \prod_{\emptyset \neq S \subset \nl} \Sigma^\infty \bigwedge_{i \in S} X_i \]
and in particular if $X$ is well-based then
\[ \Sigma^\infty X^n \simeq \bigvee_{i=1}^n \binom{n}{i} \Sigma^\infty X^{\sma i} \]
\end{cor}

We are now in a position to prove the existence of $P_n F$ for functors from retractive spaces into spectra.
First we will give a result that motivates the construction.

\subsection{An equivalence between $\cat G_n^\op$ and $\cat M_n^\op$}

Let $\cat G_n = \mc T_n$ be the category of based sets $\underline{0}_+$, $\ldots$ , $\nl_+$ and based maps between them.
The opposite of $\cat G_n$ is a full subcategory of Segal's category $\Gamma$.
As before, let $\cat M_n$ be the category of unbased sets $\underline{0} = \emptyset$, $\uline{1}$, $\ldots$ , $\nl$ and \emph{surjective} maps between them.
If $\cat I$ is a category then let $[\cat I,\mc Sp]$ denote the homotopy category of diagrams of spectra indexed by $\cat I$.

The maps in $\cat G_n$ are generated by inclusions, collapses, rearrangements, and maps that fold two points into one.
From the last section, a diagram of spectra indexed by $\cat G_n$ will split into a sum of cross effects.
The first two classes of maps (inclusions and collapses) will simply include or collapse these summands.
Therefore our diagram has redundancies.
If we throw out the redundancies, only the last two classes of maps (rearrangements and folds) still carry interesting information.
But these are exactly the maps that generate the smaller category $\cat M_n$.
We have just given a heuristic argument for the following:

\begin{prop}\label{pirash}
There is an equivalence of homotopy categories
\[ [\cat G_n,\mc Sp] \overset{C}\ra [\cat M_n,\mc Sp] \]
obtained by taking cross-effects
\[ CF(\il) = \cross_i F(\underline{1}_+,\ldots,\underline{1}_+) \]
Its inverse is obtained by taking sums
\[ [\cat G_n,\mc Sp] \overset{P}\la [\cat M_n,\mc Sp] \]
\[ PG(\il_+) = \bigvee_{j=0}^i \binom{i}{j} G(j) \]
There is also an equivalence of homotopy categories
\[ [\cat G_n^\op,\mc Sp] \simeq [\cat M_n^\op,\mc Sp] \]
obtained from co-cross effects and products
\[ CF(\il) = \cocross_i F(\underline{1}_+,\ldots,\underline{1}_+) \]
\[ PG(\il_+) = \prod_{j=0}^i \binom{i}{j} G(j) \]
\end{prop}

\begin{rmk}
The author learned a version of this result from Greg Arone.
A similar result for diagrams of abelian groups was done by Pirashvili \cite{pirashvili2000dold}.
Helmstutler \cite{helmstutler2008model} gives a more sophisticated treatment that handles both abelian groups and spectra in the same uniform way.
He gives a Quillen equivalence between the two categories of diagrams with the projective model structure.
This is of course stronger than just an equivalence of homotopy categories, but we may think of the above result as a very explicit description of the derived functors.
This perspective was essential in making the correct guess for $P_n F$ in section \ref{mapping} above, and it motivates our proof of Thm. \ref{final_third} below.
\end{rmk}

\begin{proof}
We define diagrams that extend the above constructions on objects.
The essential ingredient is to define maps between the various cubes that show up in the definition of total homotopy fiber and cofiber found in \cite{calc2}.
These maps of cubes $I^{\il} \ra I^{\jl}$ are all generalized diagonal maps coming from maps of sets $\il \la \jl$.
Then it is easy to define a natural equivalence of diagrams $CPG \ra G$.
On the other hand, Prop. \ref{splitting} gives an equivalence $PCF(\il_+) \ra F(\il_+)$ for each object $\il_+ \in \cat G_n$, but these equivalences do not commute with the maps of $\cat G_n$.
Instead, we define an isomorphism $PCF \ra F$ in the homotopy category of diagrams.
To do this, we choose for each arrow $\il_+ \ra \jl_+$ of $\cat G_n$ a contractible space of maps
\[ PCF(\il_+) \ra F(\jl_+) \]
that agrees in a natural way with compositions, and such that on the identity arrows $\il_+ = \il_+$ we choose only equivalences
\[ PCF(\il_+) \ra F(\il_+) \]
Our chosen spaces of maps $PCF(\il_+) \ra F(\jl_+)$ end up being products of cubes, the same cubes that appear in the definition of total homotopy fiber above.
This gives the desired equivalence of homotopy categories.

The contravariant case is similar, but we will give one more detail here since it is needed in the next section.
Let $F: \cat G_n^\op \ra \mc Sp$ be a diagram.
For each map $\il_+ \la \jl_+$ in $\cat G_n$, we use the diagonal map $I^{\il} \ra I^{\jl}$ to define
\[ \bigvee_{S \subset \il} I^{\il - S}_+ \sma F(S_+) \ra \bigvee_{T \subset \jl} I^{\jl - T}_+ \sma F(T_+) \]
taking the summand for $S \subset \il$ to the summand for $f^{-1}(S) \subset \jl$.
This passes to a well-defined map on the co-cross effects of $F$, which gives the arrows of the diagram $CF$.
\end{proof}

\subsection{$P_n F$ for retractive spaces over $B$ into spectra}\label{hack}

Let us consider homotopy functors
\[ \mc R_{B,\fin}^\op \overset{F}\ra \mc Sp \]
from finite retractive spaces into spectra.
Our previous construction of $P_n F$ was roughly the same as a mapping space of diagrams indexed by $\mc U_{B,n}^\op$, the spaces under $B$ with at most $n$ points.
When $B \neq *$, this approach calls for more technology because $\mc U_{B,n}$ needs to be enriched.
However, the equivalence $[\cat G_n^\op,\mc Sp] \simeq [\cat M_n^\op,\mc Sp]$ suggests that we could just eliminate the inclusion and collapse maps in $\mc U_{B,n}$.
This leads to the category $\cat M_n$ again, which does not need to be enriched.

So we replace our diagrams
\[ \begin{array}{c}
\mc U_{B,n} \ra \mc Sp \\
\il \amalg B \leadsto X^i \\
\il \amalg B \leadsto F(\il \amalg B)
\end{array} \]
with the diagrams of co-cross effects
\[ \begin{array}{c}
\cat M_n \ra \mc Sp \\
\il \leadsto X^{\barsmash i} \\
\il \leadsto \cocross_i F(\uline{1} \amalg B,\ldots,\uline{1} \amalg B)
\end{array} \]
where $\barsmash$ is the external smash product from Def. \ref{smash}.
We are being sloppy about the existence of maps into $B^i$, but this gives enough intuition to suggest that we try the following construction on retractive simplicial sets $X_\cdot$ over $S_\cdot B$:
\begin{eqnarray*}
E_n F(X_\cdot) &=& \underset{(\il \la \jl) \in a\cat M_n^\op}\holim\, \left( \underset{\Delta_{X_\cdot^{\barsmash i}}}\holim\,\, \cocross_j F(\Delta^p \amalg B,\ldots,\Delta^p \amalg B)  \right) \\
&\simeq & \underset{\cat M_n^\op \int \Delta_{X_\cdot^{\barsmash i}}}\holim\, \cocross_i F(\Delta^p \amalg B,\ldots,\Delta^p \amalg B) \\
\end{eqnarray*}
As before, the equivalence comes from Prop. \ref{holim_twisted_split}.
Here $X_\cdot^{\barsmash i}$ is a simplicial set containing $(S_\cdot B)^i$ as a retract, whose fiber over a simplex in $(S_\cdot B)^i$ is the smash product of the fibers in $X_\cdot$.
The homotopy type of $X_\cdot^{\barsmash i}$ is homotopy invariant in $X_\cdot$ by the same argument as Prop. \ref{diagonalacycliccells} above.
As before, we extend $E_n F$ to spaces by $E_n F(X) := E_n F(S_\cdot X)$.

Each surjective map $\il \la \jl$ induces a cofibration $X_\cdot^{\barsmash i} \ra X_\cdot^{\barsmash j}$.
This determines a functor $\Delta: \cat M_n^\op \ra \cat{Cat}$ by that sends $\il$ to the category $\Delta_{X_\cdot^{\barsmash i}}$.
The diagram
\[ \cat M_n^\op \int \Delta_{X^{\barsmash i}} \overset{\cocross F}\ra \mc Sp \]
is then defined by
\[ \xymatrix{
\il, & \Delta[p] \ar[r] & X_\cdot^{\barsmash i} \ar[d] &\leadsto & \cocross_i F(\Delta^p \amalg B,\ldots,\Delta^p \amalg B) \ar[d] \\
\jl, \ar[u] & \Delta[q] \ar[r] \ar[u] & X_\cdot^{\barsmash j} &\leadsto & \cocross_j F(\Delta^q \amalg B,\ldots,\Delta^q \amalg B) } \]
The map of co-cross effects is defined in the proof of Prop. \ref{pirash} above.
We can show that $E_n F$ is $n$-excisive by proving properties (1), (2), and (3) from section \ref{secondconstruction}.
Property (1) follows from the above equivalences, and property (3) is straightforward.
We will do (2) in detail.

As before, we can start with a pushout cube of simplicial sets, with initial vertex $A \in \cat{sSet}_{S_\cdot B}$.
It's indexed by a set $S$, so for each element $s \in S$, let $X_s \in \cat{sSet}_{S_\cdot B}$ be a simplicial set containing $A$ (and also containing $S_\cdot B$ as a retract).
Then there is a pushout cube which assigns each subset $T \subset S$ to the simplicial set $\bigcup_{t \in T} X_t$, which is shorthand for the pushout of the $X_t$ along $A$.
We want to show that $E_n F$ takes this to a Cartesian cube; in other words, the natural map
\[ \underset{\Delta[p] \rightarrow (\bigcup_S X_s)^{\barsmash i}}\holim\, \cocross_i F(\Delta^p \amalg B,\ldots,\Delta^p \amalg B) \hspace{50pt} \]
\[ \hspace{50pt} \ra \underset{(T \subsetneq S)^\op}\holim\, \left(\underset{\Delta[p] \rightarrow (\bigcup_T X_s)^{\barsmash i}}\holim\, \cocross_i F(\Delta^p \amalg B,\ldots,\Delta^p \amalg B) \right) \]
is an equivalence.
Using dual Thomason, we rewrite the right-hand side as
\[ \underset{(T,\il, \Delta[p] \rightarrow (\bigcup_T X_s)^{\barsmash i})}\holim\, \cocross_i F(\Delta^p \amalg B,\ldots,\Delta^p \amalg B) \]
where each object of the indexing category is a proper subset $T \subsetneq S$, integers $p \geq 0$ and $0 \leq i \leq n$, and a map $\Delta[p] \rightarrow (\bigcup_T X_s)^{\barsmash i}$.
A map between two objects looks like
\[ \xymatrix{
T, & \il, & \Delta[p] \ar[r] & (\bigcup_T X_s)^{\barsmash i} \ar[d] \\
&&& (\bigcup_T X_s)^{\barsmash j} \\
U, \ar[uu]^*\txt{subset} & \jl, \ar[uu] & \Delta[q] \ar[r] \ar[uu] & (\bigcup_U X_s)^{\barsmash j} \ar[u] } \]
As before, this category maps forward into $\cat M_n^\op \int \Delta_{(\bigcup_S X_s)^{\barsmash i}}$, in which a map between two objects looks like
\[ \xymatrix{
& \il, & \Delta[p] \ar[r] & (\bigcup_S X_s)^{\barsmash i} \ar[d] \\
& \jl, \ar[u] & \Delta[q] \ar[r] \ar[u] & (\bigcup_S X_s)^{\barsmash j} } \]
This functor $\alpha$ forgets the data of $T$ and includes $(\bigcup_T X_s)^{\barsmash i}$ into $(\bigcup_S X_s)^{\barsmash i}$.
The natural map of homotopy limits defined above is again induced by a pullback of diagrams along $\alpha$, so we just have to show that $\alpha$ is homotopy initial.
Given an object $\Delta[q] \overset\varphi\ra (\bigcup_S X_s)^{\barsmash j}$ in the target category, the overcategory $(\alpha \downarrow \varphi)$ has as its objects the factorizations of $\varphi$
\[ \Delta[q] \ra \Delta[p] \ra (\bigcup_T X_s)^{\barsmash i} \ra (\bigcup_S X_s)^{\barsmash j} \]
where $T \subsetneq S$ must be a \emph{proper} subset of $S$.

Let us give a terminal object for this overcategory.
Either the map out of $\Delta^q$ hits the basepoint section, in which case we take $T = \emptyset$, or it misses the basepoint section, in which case it gives a $j$-tuple of simplices in $\bigcup_S X_s$, each of which lands \emph{inside} one of the sets $X_s$ in the pushout.
Therefore there is a smallest subset $T \subset S$ such that $\Delta[q] \overset\varphi\ra (\bigcup_S X_s)^{\barsmash j}$ lands inside $(\bigcup_T X_s)^{\barsmash j}$, and since $j \leq n < |S|$, this subset is proper.
This gives a terminal object for the overcategory $(\alpha \downarrow \varphi)$, so it is contractible, which finishes (2).

We might now expect that $F \ra E_n F$ is an equivalence on $\mc R_{B,n}^\op$.
This turns out to be false, but Corollary \ref{dumb} suggests the following fix.
Define a new functor
\[ P_n F(X) = \overline{E_n F}(X) \times F(\uline{0} \amalg B) \]
Note that $P_n F(X)$ is $n$-excisive because it is a homotopy limit of $n$-excisive functors.

Now let $X = \jl \amalg B$ with $j \leq n$.
Then $X^{\barsmash i} \cong (\jl)^i \amalg B^i$.
We can partition $\Delta_{X^{\barsmash i}}$ into two categories, one in which the simplex lands in the basepoint section $B^i$ and another in which the simplex misses the basepoint section.
This leads to a partition of $\cat M_n^\op \int \Delta$ into three categories, one in which there are no simplices, one in which the simplices land in $B$, and one in which the simplices miss $B$.
The homotopy limit of the first two is $E_n F(\uline{0} \amalg B)$, which contains the homotopy limit of the first $F(\uline{0} \amalg B)$.
The homotopy limit of the last category is therefore $\overline{E_n F}(\jl \amalg B)$.

This last category has as its objects $\il \times \Delta[p] \ra \jl$ for varying $1 \leq i \leq n$ and $p \geq 0$, with only surjective maps allowed in the $\il$ variable and only faces for $\Delta[p]$.
This category is further partitioned into one component for each possible image of $\il$ in $\jl$, which may be represented uniquely by an order-preserving inclusion $\il \ra \jl$.
In each of these components, there is a homotopy initial subcategory of objects in which the map of sets $\il \ra \jl$ is the chosen order-preserving inclusion, and this subcategory has contractible nerve.
In total, we get a splitting of the homotopy limit
\[ \overline{E_n F}(\jl \amalg B) \simeq \prod_{\emptyset \neq \il \subset \jl} \cocross_i F(\uline{1} \amalg B,\ldots,\uline{1} \amalg B) \]
\[ P_n F(\jl \amalg B) \simeq \prod_{\il \subset \jl} \cocross_i F(\uline{1} \amalg B,\ldots,\uline{1} \amalg B) \]
Using our splitting result (Prop. \ref{splitting}), this shows that $F(\jl \amalg B) \ra P_n F(\jl \amalg B)$ is an equivalence.
This finishes the proof that $P_n F$ exists for $F$ from retractive spaces over $B$ into spectra:

\begin{thm}\label{final_third}
If $F: \mc R_{B,\fin}^\op \ra \mc Sp$ is a homotopy functor, then there is a universal $n$-excisive approximation $P_n F$, and $F \ra P_n F$ is an equivalence on spaces with at most $n$ points.
\end{thm}

\section{A convergence result}\label{convergence}

Let $F$ be a homotopy functor. Consider its tower of universal approximations
\[ F(X) \ra \ldots \ra P_n F(X) \ra \ldots \ra P_2 F(X) \ra P_1 F(X) \ra P_0 F(X), \]
We say this tower \emph{converges} at the space $X$ when $F(X)$ is equivalent to the homotopy inverse limit of the tower.
In this section, we describe a condition on $F$ which guarantees that the tower converges for all spaces $X$ of dimension less than $m$.
This result is the analogue of Goodwillie and Weiss's convergence theorem from manifold calculus (\cite{goodwillie1999embeddings}, Thm 2.3), and the proof is similar.
In fact, we give a variant of the proof that does not depend on the various ways $P_n F$ is actually constructed, only its excision properties.
This allows us to treat the various cases we considered in Thm 1.1 with the same argument.
Finally, this convergence criterion applies to the example of $F(X) = \Sigma^\infty_+ \Map_B(X,E)$, yielding a theorem that agrees precisely with the convergence theorems of Goodwillie calculus in the covariant slot $E$.

To begin, we describe how to modify definition of an analytic functor from \cite{goodwillie1999embeddings}, Def 2.1, to the continuous setting.
Recall that an $(n+1)$-cube of spaces or spectra is \emph{$k$-Cartesian} if the map from the initial vertex to the homotopy limit of the rest is $k$-connected.
\begin{df}
\begin{itemize}
\item Let $\{ X_T : T \subseteq S \}$ be a pushout cube indexed by a set $S$. For each $s \in S$, replace the map $X_\emptyset \ra X_{\{s\}}$ by a relative CW-complex with minimal dimension, and call that dimension $d_s$ if it is not infinite.
\item A contravariant homotopy functor $F$ is \emph{$m$-analytic with excess $c$} if it sends any pushout $S$-cube with all $d_s$ finite to a $k$-Cartesian cube, where
\[ k = c + \sum_{s \in S} (m - d_s) \]
\end{itemize}
\end{df}

\begin{rmk}
The condition of being $m$-analytic gets stronger as $m$ increases.
As mentioned above, this definition is the same whether the domain of $F$ is retractive spaces or unbased spaces, or whether the output of $F$ is in based spaces or spectra.
\end{rmk}

We establish the reasonableness of this definition first with an example.
\begin{prop}
If $E \ra B$ is an ex-fibration which is $m$-connected, so that the fibers $E_b$ are $(m-1)$-connected, then the contravariant functor $\mc R_B^\op \ra \mc Sp$
\[ F(X) = \Sigma^\infty_+ \Map_B(X,E) \]
is $m$-analytic with excess 0.
\end{prop}

\begin{proof}
Let $\{ X_T : T \subseteq S \}$ be a pushout cube as above with $|S| = n+1$.
Then we prove that the restriction map
\[ \Map_B(X_{\{s\}},E) \ra \Map_B(X_{\emptyset},E) \]
is $(m-d_s)$-connected by induction on $d_s$.
In each step we attach cells of dimension $d$ to the $d-1$ skeleton $X^{(d-1)}_{\{s\}}$, giving the map
\[ \Map_B(X_{\{s\}}^{(d)},E) \ra \Map_B(X_{\{s\}}^{(d-1)},E) \]
We prove this map is $(m-d)$-connected by producing a lift in each square
\[ \xymatrix{
S^{k-1} \ar[r] \ar[d] & \Map_B(X_{\{s\}}^{(d)},E) \ar[d] \\
D^k \ar[r] \ar@{-->}[ur] & \Map_B(X_{\{s\}}^{(d-1)},E) } \]
when $k \leq (m-d)$. This reduces to finding collection of lifts in squares of the form
\[ \xymatrix{
S^{d+k-1} \ar[r] \ar[d] & E \ar[d] \\
D^{d+k} \ar[r] \ar@{-->}[ur] & B } \]
but this is always possible when $d+k \leq m$, because $E \ra B$ is an $m$-connected fibration.

Now when we apply $\Map_B(-,E)$ to our pushout $(n+1)$-cube, we get a strongly Cartesian cube.
We have just proven that the final legs of this cube are $(m-d_s)$-connected maps, for varying $s \in S$.
Applying Theorem 2.4 from \cite{calc2}, we conclude that this cube is also $[\sum_{i=1}^{n+1} (m-d_s) + n-1]$-co-Cartesian.
Taking suspension spectra, we get a cube that is also $[\sum_{i=1}^{n+1} (m-d_s) + n-1]$-co-Cartesian.
By Remark 1.19 from \cite{calc2}, such a cube of spectra must also be $\sum_{i=1}^{n+1} (m-d_s)$-Cartesian.
This proves that $F$ is $m$-analytic with excess 0.
\end{proof}

The upshot of this definition is the following theorem, which gives a connectivity estimate for $F(X) \ra P_{n-1} F(X)$ when $X$ is a finite CW complex. As before, when we work in the category of retractive spaces over $B$, we take the term ``CW complex'' to mean a relative CW complex $B \ra X$.

\begin{thm}
If $F$ is $m$-analytic with excess $c$, and $X$ is a finite CW complex of dimension $d$, then the natural map
\[ F(X) \ra P_{n-1} F(X) \]
is $(n(m-d)+c)$-connected.
\end{thm}

\begin{cor}
If $F$ is $m$-analytic, then its tower converges on all finite CW complexes $X$ of dimension at most $m-1$. If in addition $F$ is finitary, then its tower converges on all CW complexes of dimension at most $m-1$.
\end{cor}

\begin{rmk}
The exact connectivity estimate given in the theorem may not extend to the case of infinite complexes. One difficulty is that it is possible for a filtered homotopy limit of $k$-connected maps to be less than $k$-connected.
\end{rmk}

\begin{proof}
Fix $n \geq 0$. We will prove the above theorem by induction on $d$.
We will do the case where the output of $F$ is based spaces, but for spectra the argument is similar and slightly easier.

If $d = 0$, then the complex $X$ is just a finite set of $k$ points added to the initial object of the domain category, which we abbreviate as $\kl$, even though it may instead be a space of the form $\kl \amalg B$.
We wish to show that
\[ F(\kl) \ra P_{n-1} F(\kl) \]
is $(nm + c)$-connected.
If $k \leq n-1$, then we already know that $F(\kl) \ra P_{n-1} F(\kl)$ is an equivalence.
So assume that $k \geq n$.
In this case, we factor the desired map as
\[ F(\kl) \simar P_k F(\kl) \ra P_{k-1} F(\kl) \ra \ldots \ra P_n F(\kl) \ra P_{n-1} F(\kl) \]
To show that the composite is $(nm + c)$-connected, it suffices to prove that
\begin{equation}\label{conv_map}
P_n F(\kl) \ra P_{n-1} F(\kl)
\end{equation}
is $(nm + c)$-connected, for then the same argument shows that the higher links of this chain are more highly connected.
Certainly when $k = n$, this map can be identified with the map
\begin{equation}\label{conv_map_2}
F(\nl) \ra \underset{S \subset \nl}\holim\, F(S)
\end{equation}
which is $c + \sum_{i=1}^{n} (m)$-connected by the analyticity condition of $F$.
When $k > n$, we form a $k$-dimensional pushout cube whose initial vertex is the inital object of the category, and each initial leg adjoins a single point, so that the final vertex is our space $\kl$.
Apply $P_n F \ra P_{n-1} F$ to this cube to get a map of two $k$-dimensional cubes.
We need to prove that the map on the initial vertices $P_n F(\kl) \ra P_{n-1} F(\kl)$ is $k$-connected, so we pick any basepoint in $P_{n-1} F(\kl)$ and examine the homotopy fiber over that point.
That point has images in the rest of the $P_{n-1} F$-cube, and taking homotopy fibers over all these points, we get a new homotopy fiber cube $\mathbf{F}$.

We are reduced to showing that the initial vertex of $\mathbf{F}$ is $(nm+c-1)$-connected.
The vertices of $\mathbf{F}$ are parametrized by the subsets of $\kl = \{1,\ldots,k\}$, so that the inital vertex is the whole set and the final vertex is empty.
By the universal properties of $P_n F$ and $P_{n-1} F$, every subset $S$ of size at most $n-1$ corresponds to a vertex whose space is contractible.
Each subset $T$ of size $n$ corresponds to a space $X_T$ which is the homotopy fiber of the map (\ref{conv_map_2}), so by the analyticity condition on $F$, $X_T$ is $(nm+c-1)$-connected.
Finally, we know that each $(n+1)$-dimensional face of $\mathbf{F}$ is Cartesian, since the same was true for both the $P_n F$ and $P_{n-1} F$ cubes.

Now map $\mathbf{F}$ forward to a new cube $\mathbf{F'}$ which at each subset $S  \subset \kl$ is the product
\[ \prod_{T \subseteq S,\, |T| = n} X_T \]
and whose maps are projections.
We prove that this new cube also has all $(n+1)$-dimensional faces Cartesian, in three steps.
First, $\mathbf{F'}$ is a fibration cube, because for every $S \subseteq \kl$ the map
\[ \prod_{T \subseteq S,\, |T| = n} X_T \ra \underset{R \subsetneq S}\lim\, \prod_{T \subseteq R,\, |T| = n} X_T \]
is either a homeomorphism or the projection $X_S \ra *$.
Second, since it is a fibration cube, the limit in the expression above is always a homotopy limit, and then when $|S| \geq n+1$ the above map is a homeomorphism.
This proves that every face of $\mathbf{F'}$ which both includes the final vertex and has dimension at least $(n+1)$ is Cartesian.
Third, by Prop. 1.7 of \cite{calc2}, this is equivalent to every face of dimension $(n+1)$ being Cartesian.
Now $\mathbf{F} \ra \mathbf{F'}$ is a map of cubes, which is an equivalence on all vertices corresponding to $S \subset \kl$ of size $\leq n$.
Since both cubes have all $(n+1)$-dimensional faces Cartesian, an easy induction proves that the map of cubes is an equivalence on every vertex, including the first.
Therefore the initial vertex of $\mathbf{F}$ is $(nm+c-1)$-connected, finishing the proof that the map (\ref{conv_map}) is $(nm+c)$-connected.

This finishes the case of 0-dimensional complexes, so we proceed with the inductive argument for higher dimensions.
Take any $d \geq 1$, and assume that for all finite $(d-1)$-dimensional complexes $Y$ the map
\[ F(Y) \ra P_{n-1} F(Y) \]
is $(n(m-d+1) + c)$-connected.
Let $X$ be a finite $d$-dimensional complex with $(d-1)$-skeleton $X^{(d-1)}$.
Then $X$ is the colimit of the $n$-dimensional pushout cube
\[ T \subset \{1,\ldots,n\} \quad \leadsto \quad X^{(d-1)} \cup_{\amalg S^{d-1}} \amalg L^d_T \]
Here $L^d_T$ is the $n$-sheeted layer cake space defined in the proof of Prop \ref{polys_determined_by_nplusone_points}, and the disjoint union is indexed by the $d$-dimensional cells of $X$.
For example, when $n = 2$ this gives the pushout square
\[ \xymatrix{
X^{(d-1)} \cup_{\amalg S^{d-1}} \amalg L^d_\emptyset \ar[r] \ar[d] & X^{(d-1)} \cup_{\amalg S^{d-1}} \amalg L^d_{\{1\}} \ar[d] \\
X^{(d-1)} \cup_{\amalg S^{d-1}} \amalg L^d_{\{2\}} \ar[r] & X^{(d-1)} \cup_{\amalg S^{d-1}} \amalg L^d_{\{1,2\}} = X
} \]
Applying $F \ra P_{n-1} F$ gives a map of two $n$-cubes.
Remove the initial vertices to get a map of two ``punctured cubes'' $W \ra Z$, which by inductive assumption is $(n(m-d+1) + c)$-connected on every vertex.
The map $\holim\, W \ra \holim\, Z$ therefore has connectivity
\[ (n(m-d+1) + c) - (n-1) = n(m-d) + c + 1 \]
Now consider the following commuting square:
\[ \xymatrix{
F(X) \ar[d] \ar[r] & \holim\, W \ar[d] \\
P_{n-1} F(X) \ar[r]^-\sim & \holim\, Z } \]
We already know that the right-hand map is $(n(m-d) + c + 1)$-connected.
By the analyticity property of $F$, the top map is $(n(m-d) + c)$-connected, so the composite of the two is $(n(m-d) + c)$-connected.
However the bottom map is an equivalence by excision, so we conclude that $F(X) \ra P_{n-1} F(X)$ is $(n(m-d) + c)$-connected, completing the induction.
\end{proof}

\bibliographystyle{amsalpha}
\bibliography{tower}{}

Department of Mathematics \\
University of Illinois at Urbana-Champaign \\
1409 W Green St \\
Urbana, IL 61801 \\
\texttt{cmalkiew@illinois.edu}

\end{document}